\documentclass[a4paper,12pt,]{article}
\usepackage{}
\usepackage{hyperref}
\usepackage{mathrsfs}
\usepackage{amsthm}
\usepackage{mathrsfs}
\usepackage{amsfonts}
\usepackage{amsmath}
\usepackage{amssymb}
\usepackage{amsthm}
\usepackage{enumerate}
\usepackage[mathscr]{eucal}
\usepackage{eqlist}
\usepackage{graphicx}
\usepackage{esint}
\usepackage{color}
\usepackage{yhmath}

\usepackage[body={15.6true cm, 25.6true cm}]{geometry}

\newtheorem{Theorem}{Theorem}[section]

\newtheorem{Lemma}[Theorem]{Lemma}
\newtheorem{Corollary}[Theorem]{Corollary}
\newtheorem{Remark}[Theorem]{Remark}

\numberwithin{equation}{section} \allowdisplaybreaks
\usepackage{mathdots}

\renewcommand\abstract{{\bf Abstract}}

\begin{document}
\title{On the number and geometric location of critical points of solutions to a semilinear elliptic equation in annular domains \footnote{\footnotesize The work is supported by National Natural Science Foundation of China (No.12001276, No.12071219, No.12090023).}}

\author{Haiyun Deng$^{1}$, Hairong Liu$^{2}$, Xiaoping Yang$^{3}$\footnote{E-mail: hydeng@nau.edu.cn(H. Deng), hrliu@njust.edu.cn(H. Liu), xpyang@nju.edu.cn(X. Yang)}\\[12pt]
 \emph {\scriptsize $^{1}$Department of Applied Mathematics, Nanjing Audit University, Nanjing, 211815, China;}\\
 \emph {\scriptsize $^{2}$School of Mathematics and Statistics, Nanjing University of Science and Technology, Nanjing, 210094, China;}\\
 \emph {\scriptsize $^{3}$Department of Mathematics, Nanjing University, Nanjing, 210093, China}}
\date{}
\maketitle

\renewcommand{\labelenumi}{[\arabic{enumi}]}

\begin{abstract}{\bf:}{\footnotesize
   ~In this paper, one of our aims is to investigate the instability of the distribution of the critical point set $\mathcal{C}(u)$ of a solution $u$ to a semilinear equation with Dirichlet boundary condition in the planar annular domains. Precisely, we prove that $\mathcal{C}(u)$ in an eccentric circle annular domain, or a petal-like domain, or an annular domain where the interior and exterior boundaries are equally scaled ellipses contains only finitely many points rather than a Jordan curve. This result indicates that the critical point set $\mathcal{C}(u)$ is unstable when any boundary of planar concentric circle annular domain $\Omega$ has some small deformation or minor perturbation. Based on studying the distribution of the nodal sets $u^{-1}_\theta(0)(u_\theta=\nabla u\cdot \theta)$ and $u^{-1}(0)$, we prove that the solution $u$ on each symmetric axis has exactly two critical points under some conditions. Meanwhile, we further obtain that $\mathcal{C}(u)$ only has two critical points in an eccentric circle annular domain, has four critical points in an exterior petal-like domain with the exterior boundary $\gamma_E$ is an ellipse, and the maximum points are distributed on the long symmetric semi-axis and the saddle points on the short symmetric semi-axis. Moreover, we describe the geometric location of critical points of the solution $u$ by the moving plane method.}
\end{abstract}



\section{Introduction and main results}
~~~~A critical point of the function $u$ is a point in domain $\Omega$ at which the gradient $\nabla u$ of the function $u$ is the zero vector. Critical points play an important role in understanding the properties of solutions to partial differential equations not only in physical aspects but also in geometrical and analytical aspects. At an elementary level, they can help us to visualize the graph of $u$, since they are some of its notable points (local maximum, minimum, or saddle points of $u$). At a more sophisticated level, if we interpret $u$ and $\nabla u$ as a gravitational, electrostatic or velocity potential and its underlying field of force or flow, the critical points are the positions of equilibrium for the field of force or stagnation points for the flow and give information on the topology of the equipotential lines or of the curves of steepest descent (or stream lines) related to $u$ \cite{Magnanini}.

It is well known that the critical point set of a solution to Poisson equation $\triangle u=1$ with zero Dirichlet boundary condition in a planar concentric circle annular domain is exactly a circle \cite{Arango,Magnanini}. The following question is naturally raised. Assume that $u$ is a non-constant solution of the following semilinear elliptic equation with Dirichlet boundary condition
\begin{equation}\label{1.1}\begin{array}{l}
\left\{
\begin{array}{l}
\triangle u+f(u)=0~~\mbox{in}~~\Omega\subset\mathbb{R}^{2},\\
u=0 ~~\mbox{on}~~\partial\Omega,
\end{array}
\right.
\end{array}\end{equation}
which conditions on function $f$ and domain $\Omega$ guarantee the existence or absence of critical Jordan curves in the critical ponit set $\mathcal{C}(u)$?

In this paper, we concentrate ourselves on the above problem and related topics. We will show that the critical point set $\mathcal{C}(u)$ is unstable only if the domain has minor changes. Actually we prove that $\mathcal{C}(u)$ merely contains finitely many isolated points rather than a Jordan curve once the shape of the concentric circle annular domain has been any deformed. We further studies and describe the concrete number and geometric location of critical points. The number, geometric location and types of the critical points of solutions to elliptic equations in planar non-convex domains is a classic and fascinating research topic. In 1992, Alessandrini and Magnanini \cite{AlessandriniMagnanini1} considered the geometric distribution of the critical point set of solution $u\in C^1(\overline{\Omega})\cap C^2(\Omega)$ of the Dirichlet problem: $\Delta u=0~\mbox{in}~\Omega,u=a_j~\mbox{on}~\Gamma_j,j=1,\cdots,N,$ where $a_1,\cdots,a_N$ do not all coincide and $\Omega$ is a bounded open set in the plane and its boundary $\partial\Omega$ is composed of $N$ simple closed curves $\Gamma_1,\cdots,\Gamma_N,N\geq 1$ of class $C^{1,\alpha}.$ They deduced that the critical point set consists of finite isolated critical points. In 2011, Arango and G\'{o}mez \cite{Arango} studied the critical points of solution $u$ to (\ref{1.1}) under the case that $\Omega$ is a planar bounded smooth domain and $f$ is an increasing, real analytic function. They showed that the critical point set of solution $u$ is made up of finitely many isolated points and finitely many Jordan curves, and present an example with two critical points in a planar annular domain, and gave further results under the case that $\Omega$ is an annular domain whose boundaries have nonzero curvature. In 2022, Gladiali and Grossi \cite{GladialiGrossi} constructed a class of non-convex planar simply connected domains such that the solution $u$ of Poisson equation $-\Delta u=1$  with zero Dirichlet boundary condition can have any number of maximum points.  Recently, in \cite{Deng2,Deng3}, we  investigated the distribution of critical and singular points of solutions $u$ to a quasilinear elliptic equations and a kind of linear elliptic equations with nonhomogeneous Dirichlet boundary conditions in a multiply connected domain $\Omega$ in $\mathbb{R}^2$, respectively.

 In higher dimensional space, as far as we know, there are few works about the geometric distribution of critical points of solutions to elliptic equations. Under the assumption of the existence of a semi-stable solution $u$ to problem (\ref{1.1}), Cabr\'{e} and Chanillo \cite{CabreChanillo} proved that the positive solution $u$ has a unique non-degenerate critical point in bounded smooth convex domains of revolution with respect to an axis in $\mathbb{R}^{n}(n\geq3)$. In 2017, Alberti, Bal and Di Cristo \cite{Alberti} investigated the existence of critical points for solutions to second-order elliptic equations of the form $\nabla\cdot(\sigma(x)\nabla u)=0$ over a bounded domain with prescribed boundary conditions in $\mathbb{R}^{n}(n\geq3).$ In 2019, Deng, Liu and Tian \cite{Deng1} considered the geometric distribution of critical points of solution to mean curvature equations with Dirichlet boundary condition on some special symmetric domains in $\mathbb{R}^{n}(n\geq3)$ by the projection method. In 2020, Grossi and Luo \cite{GrossiLuo} studied the number and geometric location of critical points of positive solutions of semilinear elliptic equations  $\triangle u+f(u)=0$ with $u_{\partial\Omega_\varepsilon}=0$ in domains $\Omega_\varepsilon=\Omega\setminus B(P,\varepsilon)$, where $\Omega\subset\mathbb{R}^n(n\geq 3)$ is a smooth bounded domain. In fact, they mainly settled the relationship between the number and geometric location of critical points of solutions in domain $\Omega_\varepsilon$ and domain $\Omega$. For other related results see \cite{Alessandrini,AlessandriniMagnanini2,Cheeger,Chen,De,Deng4,GrossiIanni,Jerison,Lin2022,Naber,Pacella} and references therein.

In this paper, we show the number and geometric location of the critical point sets $\mathcal{C}(u)$ of solutions to a semilinear elliptic equation in three kinds of planar annular domains. Before stating the main results, we give some necessary definition.

\vspace{0.2cm}
\noindent {\bf Definition 1.} {\it
 (exterior petal-like domain) We say that $\Omega\subset\mathbb{R}^2$ is an exterior petal-like domain, if $\Omega$ is a concentric annular domain with the interior boundary $\gamma_I$ and the exterior boundary $\gamma_E$ such that $\gamma_E$ is not a circle, and satisfies that\\
 (1) exterior boundary $\gamma_E$ has $k\geq 2$ symmetric axes and these $k$ symmetric axes are equiangular, denote by $l_1, l_2, \cdots, l_k$;\\
 (2) $\Omega_{\gamma_E}$ is $l_i$-convex for any $i=1,2,\cdots,k$, where $\Omega_{\gamma_E}$ is the domain bounded by the exterior boundary $\gamma_E$; \\
 (3) exterior boundary $\gamma_E$ satisfies the inner cone condition;\\
 (4) the interior boundary $\gamma_I$ is a circle whose center coincides with the center of $\Omega_{\gamma_E}$.}
 \vspace{0.1cm}

Similarly, we can define the interior petal-like domain. In this paper, all the conclusions on an exterior petal-like domain are also valid for an interior petal-like domain, and the proof method is similar, which will not be explained later.  For the instability of the critical point sets, we have the following result.

\vspace{0.2cm}
\noindent {\bf Theorem A.} {\it
Let $u$ be a non-constant solution to (\ref{1.1}). Assume that $\Omega\subset \mathbb{R}^2$ is one of the following domains: an eccentric circle annular domain, an exterior petal-like domain and an annular domain where the interior and exterior boundaries are equally scaled ellipses. If $f(u)\geq 0$ is a decreasing and real analytic function, then the critical point set $\mathcal{C}(u)$ of $u$ does not contain any Jordan curves.}
\vspace{0.2cm}

\vspace{0.2cm}
\noindent{\bf Theorem B.} {\it
Let $u$ be a non-constant solution to (\ref{1.1}), $f(u)\geq 0$ be a decreasing and real analytic function. Assume that $\Omega\subset\mathbb{R}^2$ is one of the following three kinds of domains:\\
  (i) an eccentric circle annular domain; \\
  (ii) a smooth exterior petal-like domain with even number axes of symmetry satisfied that the symmetric axes are perpendicular in pairs and the exterior boundary has nonzero curvature; \\
  (iii) an annular domain where the interior and exterior boundaries are both ellipses. \\
  Then on each symmetric axis $u$ has exactly two critical points.}
\vspace{0.2cm}

 For the exact number of the critical points of the solution $u$ in domain $\Omega$, our results are as follows:

\vspace{0.2cm}
\noindent{\bf Theorem C.} {\it
Let $u$ be a non-constant solution to (\ref{1.1}), $f(u)\geq 0$ be a decreasing and real analytic function. Assume that $\Omega\subset\mathbb{R}^2$ is one of the following two kinds of domains:\\
  (i) an eccentric circle annular domain; \\
  (ii) an exterior petal-like domain with the exterior boundary $\gamma_E$ is an ellipse. \\
  Then $u$ has exactly two critical points on each symmetric axis, does not have any extra critical points on other sub-domains, and the maximum points are distributed on the long symmetric semi-axis and the saddle points on the short symmetric semi-axis.}
\vspace{0.2cm}

The next theorem describe the geometric location of critical points of a solution $u$ in a particular exterior petal-like domain or an eccentric circle annular domain.

\vspace{0.2cm}
\noindent {\bf Theorem D.}
{\it Let $u$ be a non-constant solution to (\ref{1.1}) and $f(u)\geq 0$ be locally Lipschitz.\\
 {\bf (1)} Suppose that $\Omega\subset\mathbb{R}^2$ is an exterior petal-like domain with the exterior boundary $\gamma_E$ is an ellipse, where the radius of the interior boundary $\gamma_I$ is $a$, and the length of each semi-axis of $\gamma_E$ is $b_i$ ($i=1,2$) respectively. Then all the critical points of the solution $u$ are located within the domain
\begin{eqnarray*}\label{}
\big\{(-\frac{a+b_1}{2},\frac{a+b_1}{2})\times(-\frac{a+b_2}{2},\frac{a+b_2}{2})\big\}\cap\Big\{\Omega\setminus\big\{(x,y):a<\sqrt{x^2+y^2}\leq \sqrt{ac}\big\}\Big\},
\end{eqnarray*}
where $c$ is the radius of inscribed circle of $\Omega_{\gamma_E}$, and $\Omega_{\gamma_E}$ is the domain bounded by the exterior boundary $\gamma_E$.\\
 {\bf  (2)} Suppose that $\Omega\subset \mathbb{R}^2$ is an eccentric circle annular domain, where the interior boundary $\gamma_I$ is a circle with its radius $r$ centered at $(a,0)(a>0)$, and the exterior boundary $\gamma_E$ is a circle with its radius $R$ centered at $(0,0)$ such that $a+r<R$. Then all the critical points of the solution $u$ are located within the domain
\begin{eqnarray*}\label{2.5}
\Big\{(\frac{-R+a-r}{2},\frac{a+r+R}{2})\times(-R,R)\Big\}\cap\\
\Big\{\Omega\setminus\big\{(x,y):r<\sqrt{(x-a)^2+y^2}\leq \sqrt{r(R-a)}\big\}\Big\}.
\end{eqnarray*}}
\vspace{0.2cm}

Theorem $A$ indicates that the distribution of critical point set of the solution $u$ is unstable when any boundary of planar concentric circle annular domain $\Omega$ has a slight change, whether it has a small deformation or a very small displacement. The critical point set $\mathcal{C}(u)$ changes from a Jordan curve to finitely many isolated points. This conclusion is quite different from that on simply connected domains. In fact, by Theorem 1 in \cite{GidasNi} and Theorem 1.3 in \cite{ Berestycki}, we know that the critical point set $\mathcal{C}(u)$ of a solution $u$ to (\ref{1.1}) is stable under small domain deformation in a bounded Steiner-symmetric domain $\Omega\subset\mathbb{R}^n(n\geq 2)$, namely, the number, geometric location and non-degeneracy of critical points of the solution $u$ do not change. Our main argument depends on various maximum principles, and the delicate analysis about the distribution of the nodal sets $u_\theta^{-1}(0)$ ($u_\theta=\nabla u\cdot \theta$) and $u^{-1}(0)$. We obtain the exact number of the critical point sets $\mathcal{C}(u)$ of solutions to a semilinear elliptic equation in two kinds of planar annular domains. In addition, we use the interior and exterior boundary information to obtain the geometric location of the critical point set by the moving plane method and the moving sphere method, respectively.

The rest of this article is written as follows. In Section 2, we will prove Theorem A. In detail, we show the instability of the critical point sets $\mathcal{C}(u)$, namely, we prove that $\mathcal{C}(u)$ in an eccentric circle annular domain, or an exterior petal-like domain, or a planar annular domain where the interior and exterior boundaries are equally scaled ellipses contains only finitely many points rather than a Jordan curve. In Sections 3, we will prove Theorem B and Theorem C. In Sections 4, we will prove Theorem D, that is, we describe the geometric location of critical points of the solutions in a planar exterior petal-like domain with the exterior boundary $\gamma_E$ is an ellipse and an eccentric circle annular domain, respectively.

\section{Instability of critical point sets}
~~~~In this section, we will present the proof of Theorem A. In order to prove the main results, we need the following lemma. Relevant results can be seen in \cite{AlessandriniMagnanini1}, for the sake of completeness, in our setting, we will give the following complete proof.

\begin{Lemma}\label{le2.1}
Let $u$ be a non-constant solution to (\ref{1.1}) and $f(u)\geq 0$ be a decreasing and real analytic function. \\
(i) If $\Omega$ is a smooth planar domain, and the critical set $\mathcal{C}(u)$ of solution $u$ contains a Jordan curve $\gamma$ in $\Omega$, then $\Omega$ cannot be simply connected.\\
(ii) If $\Omega$ is a smooth planar annular domain, then $\mathcal{C}(u)$ has at most one Jordan curve $\gamma$, denoted $\Omega_\gamma$ by the inside of critical Jordan curve $\gamma$, and the interior boundary $\gamma_I\subset\Omega_\gamma$.
\begin{proof} (i) We set up the usual contradiction argument. Suppose that $\Omega$ is simply connected. For any direction $\alpha\in S^1$, then the directional derivative $u_\alpha:=\nabla u\cdot\alpha$ satisfies the following boundary value problem
\begin{equation}\label{2.1}\begin{array}{l}
\left\{
\begin{array}{l}
(\triangle+f'(u))u_\alpha=0~~\mbox{in}~~\Omega_\gamma,\\
u_\alpha=0 ~~\mbox{on}~~\partial\Omega_\gamma,
\end{array}
\right.
\end{array}\end{equation}
where $\Omega_\gamma$ is the domain bounded by the critical Jordan curve $\gamma$. By $f'(u)\leq 0$ and the strong maximum principle, we have that $u_\alpha\equiv 0$ in $\Omega_\gamma$ for all direction $\alpha.$ Since $f$ is real analytic, then $u$ and $u_\alpha$ are also analytic. This means that $u_\alpha\equiv 0$ for all direction $\alpha$ in $\Omega$, namely, $u\equiv C$, $C$ is a constant. By the boundary condition of (\ref{1.1}), we have that
$$u\equiv 0~\mbox{in}~\Omega.$$
This is a contradiction, that is, $\Omega$ is not simply connected.\\
(ii) Suppose by contradiction that the critical point set $\mathcal{C}(u)$ has two Jordan curves $\gamma_1,\gamma_2$. By the result of (i), denoted $\Omega_{\gamma_1}$ and $\Omega_{\gamma_2}$ by the inside of critical Jordan curves $\gamma_1$ and $\gamma_2$ respectively, we know that both $\Omega_{\gamma_1}$ and $\Omega_{\gamma_2}$ contain interior boundary $\gamma_I$. Then $\gamma_2\subset\Omega_{\gamma_1}$, $\gamma_1\subset\Omega_{\gamma_2}$, or $\gamma_1$ and $\gamma_2$ have intersections. Without loss of generality, we suppose that $\gamma_2\subset\Omega_{\gamma_1}$. For any direction $\alpha\in S^1$, then $u_\alpha=\nabla u\cdot\alpha$ satisfies the following boundary value problem
\begin{equation*}\label{}\begin{array}{l}
\left\{
\begin{array}{l}
(\triangle+f'(u))u_\alpha=0~~\mbox{in}~~\Omega_{\gamma_1}\setminus\Omega_{\gamma_2},\\
u_\alpha=0 ~~\mbox{on}~~\gamma_1~\mbox{and}~ \gamma_2,
\end{array}
\right.
\end{array}\end{equation*}
Similarly, by the strong maximum principle and the analyticity of $u$ and $u_\alpha$, we can easily obtain that $u\equiv 0~\mbox{in}~\Omega$, this completes the proof of (ii). \end{proof}
\end{Lemma}

If $\Omega$ is a smooth planar exterior petal-like domain, by Hopf Lemma, we know that the solution $u$ of (\ref{1.1}) does not have critical points on $\partial\Omega$. However, if there exists cusps on the exterior boundary $\gamma_E$, there could be critical points on the exterior boundary.

\begin{Lemma}\label{le2.2}
Let $\Omega$ be a planar exterior petal-like domain and $f(u)\geq 0$ be a decreasing and real analytic function. Then any boundary critical point of solution $u$ to (\ref{1.1}) must be a cusp. Moreover, any boundary point is not the convergence point of the critical set $\mathcal{C}(u)$.
\begin{proof} Suppose that $p\in\partial\Omega,$ we need divide the proof into two cases.

 Case 1: If $p\in\partial\Omega$ is not a cusp, by Hopf boundary lemma, we have that
$$\frac{\partial u}{\partial \nu}(p)\neq 0,~\mbox{i.e.},~\nabla u(p)\neq 0,$$
where $\nu$ is the unit outward normal vector of $\partial\Omega$ at $p$, therefore critical points cannot accumulate at $p$.

Case 2: If $p\in\partial\Omega$ is a cusp, suppose by contradiction that there exists a cluster of points $\{p_n\}\subset\mathcal{C}(u)$ with $p_n\rightarrow p.$ By the symmetry of domain $\Omega$, then for every $q\in \{p_n\}$ close enough to $p,$ there exists a straight segment $L$, such that $q\in L$,  $\Omega(L)\subset\Omega$ and $\Omega^*(L)\subset\Omega$, where $\Omega^*(L)$ is the symmetric domain of $\Omega(L)$ with respect to straight segment $L$. See Fig.1.
\begin{center}
  \includegraphics[width=2.7cm,height=3.2cm]{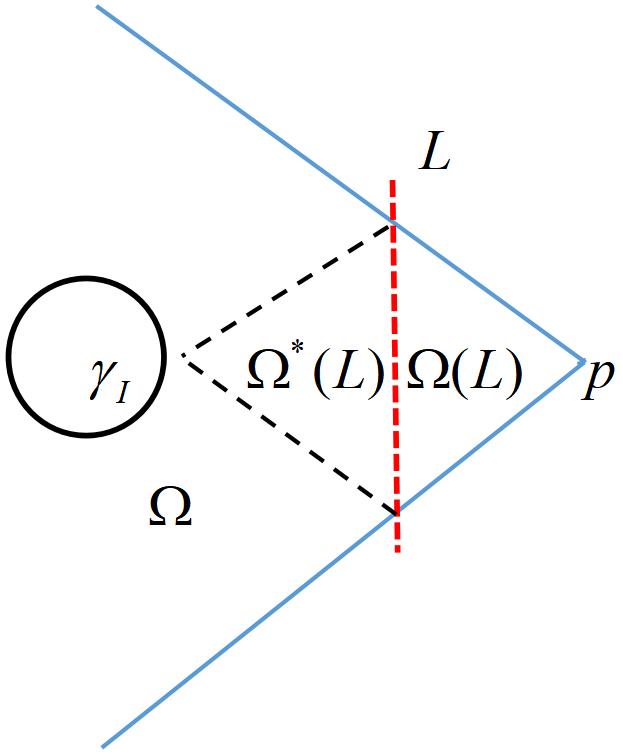}\\
  \scriptsize {\bf Fig. 1.}~~The domain around the corner $p$ of the exterior boundary.
\end{center}
 For every $x\in\Omega^*(L)$, we suppose that $V(x)=u(x)-u(x')$, where $x'$ is the the symmetric point of $x$ with respect to line straight segment $L$. Using the differential mean value theorem, then $V(x)$ satisfies the following boundary problem
\begin{equation}\label{2.2}\begin{array}{l}
\left\{
\begin{array}{l}
(\triangle+f'(\zeta(x)))V(x)=0~~\mbox{in}~~\Omega^*(L),\\
V(x)>0 ~~\mbox{on}~~\partial\Omega^*(L)\setminus L,\\
V(x)=0 ~~\mbox{on}~~L\cap\partial\Omega^*(L),
\end{array}
\right.
\end{array}\end{equation}
where $\zeta(x)$ is a smooth function. By the strong maximum principle, we know that $V(x)>0$ in $\Omega^*(L)$, and by Hopf Lemma, we have that $$\frac{\partial V(x)}{\partial\nu}<0~\mbox{on}~L\cap\partial\Omega^*(L).$$
It means that there does not exist any critical points on $L\cap\partial\Omega^*(L),$ this contradicts $q\in \{p_n\}$ on the line straight segment $L$.\end{proof}
\end{Lemma}

By Lemma \ref{le2.2}, we can easily obtain the following results.

\begin{Corollary}\label{cor2.3}
  Under the hypothesis of Lemma \ref{le2.2}, there exists a compact set $\mathbb{W}$ with two smooth boundaries such that $\mathcal{C}(u)\subset\mathbb{W}\subset\Omega$. By Theorem 3.1 \cite{Arango} and Lemma \ref{le2.2}, then the critical point set $\mathcal{C}(u)$ is made up of finitely many isolated points and at most one critical Jordan curve in $\Omega$.
\end{Corollary}

Next we will prove Theorem A in an eccentric circle annular domain, an exterior petal-like domain with the exterior boundary $\gamma_E$ is an ellipse, an exterior petal-like domain with the exterior boundary $\gamma_E$ is an equilateral triangle, and a planar annular domain where the interior and exterior boundaries are equally scaled ellipses, respectively.

\begin{proof}[\bf Proof of Theorem A]

{\bf Case $A_1$:} If $\Omega$ is an eccentric circle annular domain. We need divide the proof into two situations.

 Situation 1: $\mathcal{C}(u)$ does not self-form closed curves that does not contain the interior boundary $\gamma_I$. This can be easily obtained by Lemma \ref{le2.1}.

 Situation 2: We claim that $\mathcal{C}(u)$ also does not form a closed curve containing the interior boundary $\gamma_I$. Without loss of generality,
we assume that $\Omega\subset \mathbb{R}^2$ is an eccentric circle annular domain, where the interior boundary $\gamma_I$ is a circle with its radius $r$ centered at $(a,0)(a>0)$, and the exterior boundary $\gamma_E$ also is a circle with its radius $R$ centered at (0,0) such that $a+r<R$, see Fig. 2.
\begin{center}
  \includegraphics[width=4.5cm,height=4.5cm]{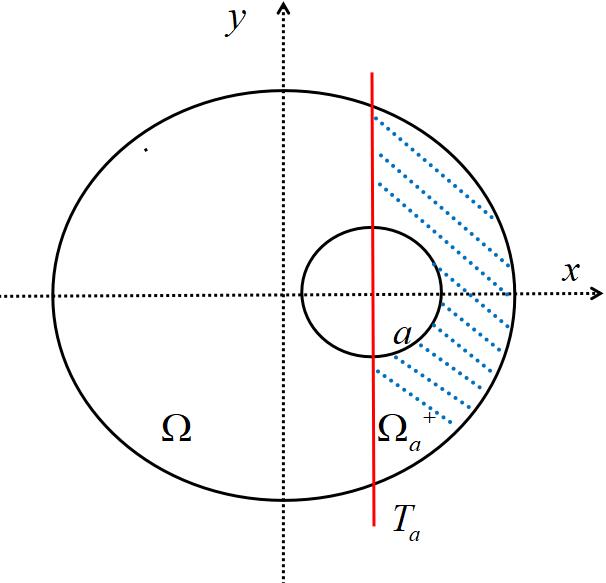}\\
  \scriptsize {\bf Fig. 2.}~~ Eccentric circle annular domain.
\end{center}

Next we denote by
\begin{equation*}\label{}\begin{array}{l}
\begin{array}{l}
T_a=\{p=(x,y)\in\Omega: x=a\},\\
\Omega_a^+=\{p=(x,y)\in\Omega: x>a\}.
\end{array}
\end{array}\end{equation*}
In $\Omega_a^+$, we set
$$v(x,y)=u(2a-x,y),$$
$$w(x,y)=u(x,y)-v(x,y)~\mbox{for~any}~x\in\Omega_a^+,$$
then we have
\begin{equation*}\label{}\begin{array}{l}
\left\{
\begin{array}{l}
\Delta w+f^\prime(\tau(x))w=0~\mbox{in}~\Omega_a^+,\\
w=0~\mbox{on}~\partial\Omega_a^+\cap T_a,\\
w<0~\mbox{on}~\partial\Omega_a^+\setminus T_a,
\end{array}
\right.
\end{array}\end{equation*}
where $\tau(x)$ is a bounded function.

By the maximum principle and Hopf lemma, we have
\begin{equation}\label{2.3}\begin{array}{l}
\left\{
\begin{array}{l}
w<0,~\mbox{in}~\Omega_a^+,\\
\frac{\partial w}{\partial\nu}|_{T_a}=-\frac{\partial w}{\partial x}|_{T_a}=-2\frac{\partial u}{\partial x}|_{T_a}>0.
\end{array}
\right.
\end{array}\end{equation}
This means that the solution $u$ can not exist critical points on $T_a,$ namely, $\mathcal{C}(u)$ can not form a Jordan curve containing the interior boundary $\gamma_I$. This completes the proof of Case $A_1$.

When $\Omega$ is an exterior petal-like domain, by the definition of exterior petal-like domain, without loss of generality, we only consider  the cases of two axes of symmetry and three axes of symmetry, that is, exterior petal-like domain with the exterior boundary $\gamma_E$ is an ellipse and exterior petal-like domain with the exterior boundary $\gamma_E$ is an equilateral triangle respectively. Other cases can be considered similarly.

{\bf Case $A_2$:} If $\Omega$ is an exterior petal-like domain with the exterior boundary $\gamma_E$ is an ellipse. In this case, we choose rotation $T=e^{i\frac{\pi}{2}}$, let $\Omega^{*}:=T(\Omega), W:=\Omega\cap\Omega^{*}$, see Fig 3.
\begin{center}
  \includegraphics[width=4.5cm,height=4.5cm]{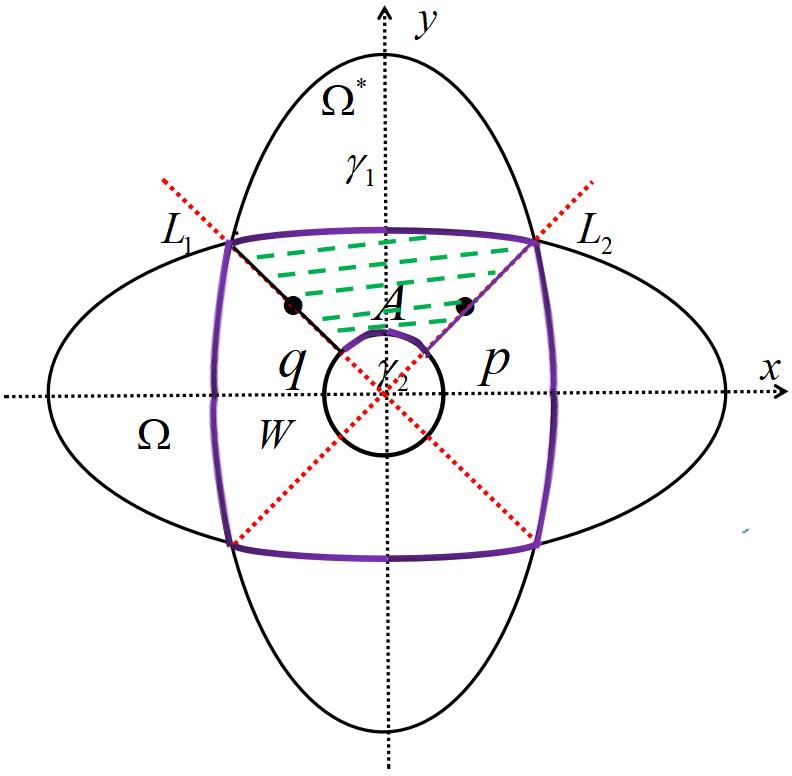}\\
  \scriptsize {\bf Fig. 3.}~~ The domain $W$ and the subdomain $A$ with the boundary $\partial A=\gamma_1\cup\gamma_2\cup L_1\cup L_2$.
\end{center}

We write $u^*(x)=u(T^{-1}(x))$ and have that $u^*(x)$ satisfies (\ref{1.1}) in $\Omega^{*}$. Note that $W$ is divided into four congruent subdomains as it is shown in Fig. 3. Without loss of generality, we consider subdomain $A$ such that $\partial A=\gamma_1\cup\gamma_2\cup L_1\cup L_2$, where $\gamma_1=\partial\Omega\cap\Omega^{*}$. $\gamma_2$ is the arc $re^{i\theta}$ where $r$ is the radius of the interior boundary $\gamma_I$, $\frac{\pi}{4}<\theta<\frac{3\pi}{4}$, $L_1$ and $L_2$ are the line segments of the rays $e^{i\frac{\pi}{4}}$ and $e^{i\frac{3\pi}{4}}$ respectively.

We remove the two end points of curve segment $\gamma_1$ such that $T^{-1}(\gamma_1)\subset\Omega.$ Since $u^{*}(\gamma_1)=u(T^{-1}(\gamma_1))$, we know that
$$u^{*}(x)>0~~\mbox{for any}~~x\in \gamma_1.$$

Furthermore, by the symmetry of domain $\Omega$ and equation (\ref{1.1}), we have that
 $$u^{*}(x)=u(x)~~\mbox{on}~~L_1\cup L_2,$$
 and
 $$u^{*}(x)=u(x)=0~~\mbox{on}~~\gamma_2.$$
 By the boundary condition, we have $u(x)-u^{*}(x)\leq 0$ on $\partial A.$ Then we have that
 \begin{equation}\label{2.4}\begin{array}{l}
\left\{
\begin{array}{l}
(\triangle+f'(\varphi(x)))(u-u^{*})=0~~\mbox{in}~~A,\\
u-u^{*}=0 ~~\mbox{on}~~\gamma_2\cup L_1\cup L_2,\\
u-u^{*}\leq 0 ~~\mbox{on}~~\gamma_1,
\end{array}
\right.
\end{array}\end{equation}
where $\varphi(x)$ is a smooth function. By the strong maximum principle, we obtain that
 $$u(x)-u^{*}(x)<0~~\mbox{in}~~A.$$

By Hopf Lemma, we have that
$$\frac{\partial (u-u^{*})(x)}{\partial\nu}>0~\mbox{on}~L_1\cup L_2.$$
It means that $\nabla(u-u^{*})(x)\neq 0$ on $L_1\cup L_2.$ Suppose by contradiction that the critical set $\mathcal{C}(u)$ contains a critical Jordan curve $\gamma$ in $\Omega$. By Lemma \ref{le2.1}, $\gamma$ must intersect the line segment $L_2$ at a point, which denoted by point $p$. It means that
\begin{equation}\label{2.5}\begin{array}{l}
\nabla u(p)=0.
\end{array}\end{equation}
 By $u^*(x)=u(T^{-1}(x))$  and $p\in L_2$, then $q:=T(p)\in L_1$ is the reflection of $p$ about the $y$ axis. In addition, according to $u(x)=u^*(T(x))$ for any $x\in\Omega,$ we have that $\nabla u^*(q)=0.$ Further using the symmetry of the domain, we know that
 \begin{equation}\label{2.6}\begin{array}{l}
\nabla u^*(p)=0.
\end{array}\end{equation}

By (\ref{2.5}) and (\ref{2.6}), we get a contradiction, therefore the critical point set $\mathcal{C}(u)$ does not contain any critical Jordan curves in $\Omega$. This completes the proof of Case $A_2$.

{\bf Case $A_3$:} If $\Omega$ is an exterior petal-like domain with the exterior boundary $\gamma_E$ is an equilateral triangle. In this case, we choose rotation $T=e^{i\frac{\pi}{3}}$, let $\Omega^{*}:=T(\Omega), U:=\Omega\cap\Omega^{*}$,
see Fig 4.
\begin{center}
  \includegraphics[width=5.0cm,height=5.0cm]{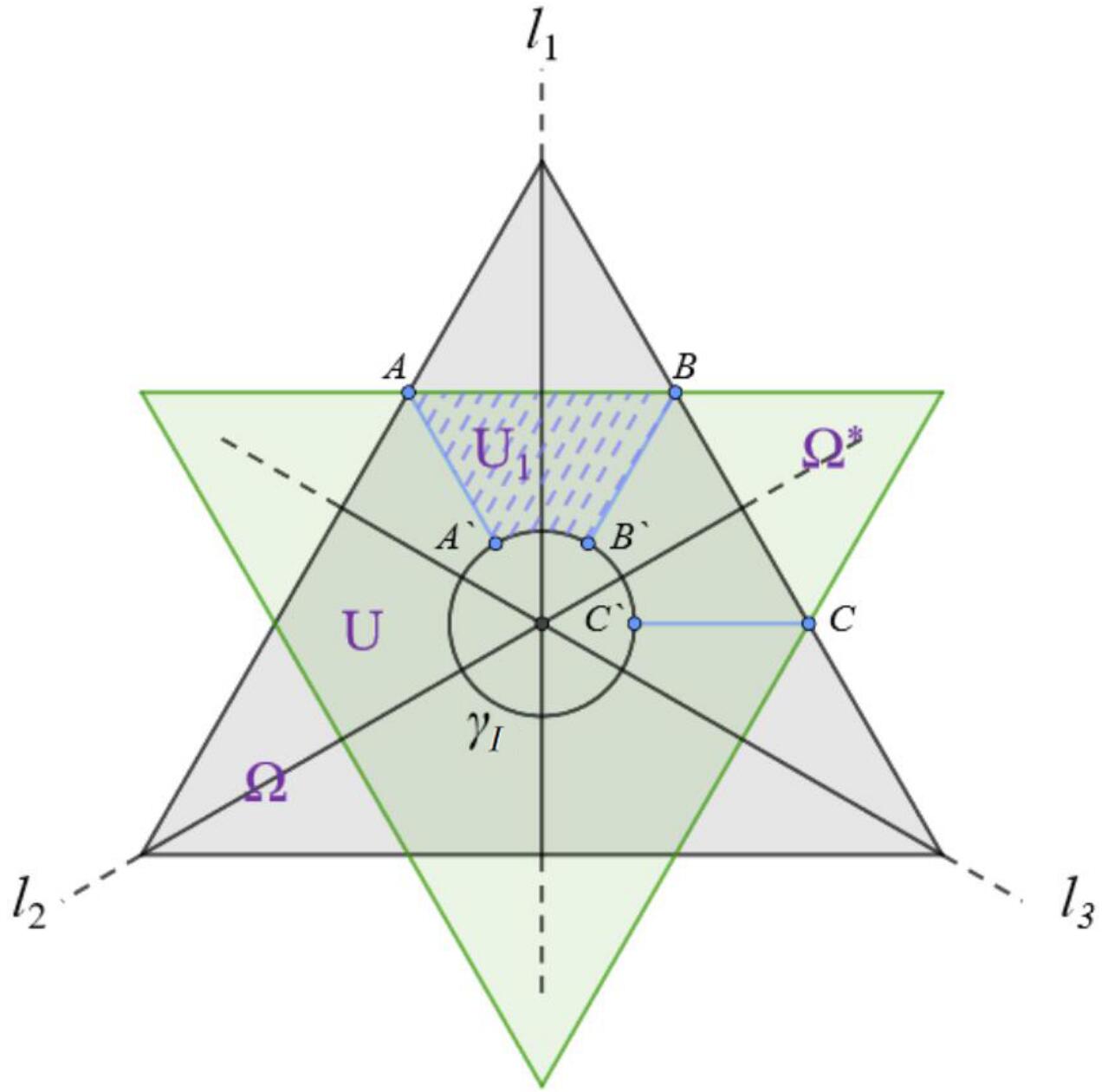}\\
  \scriptsize {\bf Fig. 4.}~~ The domain $U$ and the subdomain $U_1$ with the boundary $\partial U_1=\overline{AB}\cup\wideparen{A^\prime B^\prime}\cup \overline{AA^\prime}\cup \overline{BB^\prime}$.
\end{center}

We write $u^*(x)=u(T^{-1}(x))$ and have that $u^*(x)$ satisfies (\ref{1.1}) in $\Omega^{*}$. Note that $U$ is divided into six congruent subdomains as it is shown in Fig. 4. Without loss of generality, we consider subdomain $U_1$ such that $\partial U_1=\overline{AB}\cup\wideparen{A^\prime B^\prime}\cup \overline{AA^\prime}\cup \overline{BB^\prime}$, where $\overline{AB}=\partial\Omega^{*}\cap\Omega$, $\wideparen{A^\prime B^\prime }$ is the arc $re^{i\theta}$ and $r$ is the radius of the interior boundary $\gamma_I$, $\frac{\pi}{3}<\theta<\frac{2\pi}{3}$, $\overline{AA^\prime}$ and $\overline{BB^\prime}$ are the line segments of the rays $e^{i\frac{\pi}{3}}$ and $e^{i\frac{2\pi}{3}}$ respectively.

We remove the two end points of line segment $\overline{AB}$ such that $T^{-1}(\overline{AB})\subset\Omega^{*}\cap \partial\Omega.$ Since $u^{*}(\overline{AB})=u(T^{-1}(\overline{AB}))$, we know that
$$u^{*}(x)=0~~\mbox{for any}~~x\in \overline{AB}.$$

Furthermore, by the symmetry of domain $\Omega$ and equation (\ref{1.1}), we have that
 $$u^{*}(x)=u(x)~~\mbox{on}~~\overline{AA^\prime}\cup \overline{BB^\prime},$$
 and
 $$u^{*}(x)=u(x)=0~~\mbox{on}~~\wideparen{A^\prime B^\prime}.$$
 By the boundary condition, we have $u(x)-u^{*}(x)\geq 0$ on $\partial U_1.$ Then we have that
 \begin{equation}\label{2.7}\begin{array}{l}
\left\{
\begin{array}{l}
(\triangle+f'(\phi(x)))(u-u^{*})=0~~\mbox{in}~~U_1,\\
u-u^{*}=0 ~~\mbox{on}~~\wideparen{A^\prime B^\prime}\cup \overline{AA^\prime}\cup \overline{BB^\prime},\\
u-u^{*}\geq 0 ~~\mbox{on}~~\overline{AB},
\end{array}
\right.
\end{array}\end{equation}
where $\phi(x)$ is a smooth function. By the strong maximum principle, we obtain that
 $$u(x)-u^{*}(x)>0~~\mbox{in}~~U_1.$$

By Hopf Lemma, we have that
$$\frac{\partial (u-u^{*})(x)}{\partial\nu}<0~\mbox{on}~\overline{AA^\prime}\cup \overline{BB^\prime}.$$
It means that $\nabla(u-u^{*})(x)\neq 0$ on $\overline{AA^\prime}\cup \overline{BB^\prime}.$ According to Lemma \ref{le2.2}, we know that $\mathcal{C}(u)$ does not self-form closed curves that do not contain the interior boundary $\gamma_I$. Now suppose by contradiction that the critical set $\mathcal{C}(u)$ contains a Jordan curve $\gamma$ in $\Omega$. The rest proof of this case is similar to the proof of Case $A_2$, then the critical point set $\mathcal{C}(u)$ does not exist any Jordan curves in $\Omega$. This completes the proof of Case $A_3$.

{\bf Case $A_4$:} If $\Omega$ is a planar annular domain where the interior and exterior boundaries are equally scaled ellipses. For this case, the idea of the proof is essentially the same as Case $A_2$. Here we omit the proof. \end{proof}

\begin{Remark}\label{rem2.4}
The conclusion of Theorem $A$ is quite different from the conclusion on simply connected domains. In fact, by Theorem 1 in \cite{GidasNi} and Theorem 1.3 in \cite{Berestycki}, we know that the critical point set $\mathcal{C}(u)$ of a solution $u$ to (\ref{1.1}) is stable under small domain deformation in a bounded Steiner-symmetric domain $\Omega\subset\mathbb{R}^n(n\geq 2)$, namely, the number, geometric location and non-degeneracy of critical points of the solution $u$ do not change.
\end{Remark}

\section{Number of critical points}
~~~~Before we present the proof of the exact number of critical points, we need give some notations. For a given direction $\theta\in \mathbb{S}^1$, we define the directional derivative and its nodal set:
$$u_\theta(x)=\nabla u(x)\cdot \theta,x\in \overline{\Omega}~\mbox{and}~N_\theta=\{x\in\overline{\Omega}:u_\theta(x)=0\},$$
and the following critical point set of directional derivative $u_\theta(x)$:
$$M_\theta=\{x\in N_\theta: \nabla u_\theta(x)=D^2u(x)\cdot\theta=0\},$$
where $D^2u(x)$ denotes the Hessian matrix of $u$ at point $x.$ We know that $\mathcal{C}(u)=N_\alpha\cap N_\beta$ for any two noncollinear directions $\alpha,\beta\in\mathbb{S}^1.$

By the results of \cite{Arango,CabreChanillo}, we can give the following detail description of $N_\theta$.
\begin{Remark}\label{rem3.1}
Assume that $\Omega\subset \mathbb{R}^2$ is a smooth annular domain and the boundaries have nonzero curvature. Let $u$ be a solution to (\ref{1.1}) and $f(u)\geq 0$ be a decreasing and real analytic function. Then (1) $N_\theta$ does not self-form circles within $\Omega$ unless it contains the interior boundary $\gamma_I$; (2) if $N_\theta$ does not contain the interior boundary $\gamma_I$, then $N_\theta$ must intersect at two points with the interior and exterior boundaries respectively; (3) there are no critical points of $u_\theta(x)$ on $\partial\Omega,$ that is, $M_\theta\cap \partial\Omega=\varnothing.$
\end{Remark}

The following results due to Cheng \cite{Cheng}, provide precise description of $u^{-1}(0)$.

\begin{Lemma}\label{le3.2}
(see \cite[Theorem 2.5]{Cheng}) Let $\Omega\subset\mathbb{R}^2$ be a bounded domain. Then, for any solution of the equation $(\Delta+h(x))u=0, h\in C^\infty(\Omega)$, the following are true:\\
(1) The critical points of $u$ on the nodal lines are isolated and at regular points the nodal set is a smooth curve;\\
(2) when the nodal lines meet, their intersections are the critical points of $u$, they form an equiangular system of at least four rays splitting $\Omega$ into a finite number of connected subdomains.
\end{Lemma}

Now we will give the lower bound of number of critical points of the solution $u$ to (\ref{1.1}) in an exterior petal-like domain.

\begin{Theorem}\label{thm3.3}
 Let $u$ be a non-constant solution to (\ref{1.1}). Assume that $\Omega\subset\mathbb{R}^2$ is an exterior petal-like domain. If $f(u)\geq 0$ is a decreasing and real analytic function, then on each symmetric axis $u$ has at least two critical points.
\end{Theorem}

\begin{proof} From the assumption that $f$ is a decreasing function and the strong maximum principle, it can easily know that the solution to (\ref{1.1}) is unique. According to the symmetry of exterior petal-like domain $\Omega$, without loss of generality, let us suppose that $\Omega$ is symmetric with respect to $y$-axis. In addition, by the monotonicity of function $f$ and problem (\ref{1.1}), we know that
$$u(-x,y)=u(x,y).$$
In fact, by the symmetric assumption of domain $\Omega$, we know that $\Omega$ is symmetric with respect to the hyperplane
$$T_1=\{(x,y)\in\mathbb{R}^2,x=0\},$$
and we define that the domains $\Omega^-=\{x\in\Omega, x<0\}$ and $\Omega^+=\{x\in\Omega, x>0\}$ are the left and right of $T_1$ respectively. Next we assume that $v^-$ and $v^+$ are the reflected functions of $u$ in the domains of $\Omega^-$ and $\Omega^+$, respectively,
\begin{equation*}\label{}\begin{array}{l}
v^-(x,y):=u(-x,y) ~~\mbox{for}~(x,y)\in\Omega^-,\\
v^+(x,y):=u(-x,y) ~~\mbox{for}~(x,y)\in\Omega^+.
\end{array}\end{equation*}
By (\ref{1.1}), the symmetry of domain $\Omega$ and the differential mean value theorem, considering the functions $w^-=u-v^-$ and $w^+=u-v^+$, we have
\begin{equation}\label{3.1}\begin{array}{l}
\left\{
\begin{array}{l}
\triangle w^- +f'(\varphi^-(x))w^-=0~~\mbox{in}~\Omega^-,\\
w^-=0~~\mbox{on}~\partial\Omega^-,
\end{array}
\right.
\end{array}\end{equation}
and
\begin{equation}\label{3.2}\begin{array}{l}
\left\{
\begin{array}{l}
\triangle w^+ +f'(\varphi^+(x))w^+=0~~\mbox{in}~\Omega^+,\\
w^+=0~~\mbox{on}~\partial\Omega^+,
\end{array}
\right.
\end{array}\end{equation}
where $\varphi^-(x)$ and $\varphi^+(x)$ are two smooth functions. Using (\ref{3.1}) and (\ref{3.2}), we obtain that
\begin{equation*}\label{}\begin{array}{l}
w^-(x)=0 ~~\mbox{in}~\Omega^-,\\
w^+(x)=0 ~~\mbox{in}~\Omega^+.
\end{array}\end{equation*}
According to the definitions of $\Omega^-$ and $\Omega^+$, then we have
$$u(-x,y)=u(x,y).$$

Therefore we have that
\begin{equation}\label{3.3}\begin{array}{l}
\frac{\partial u}{\partial x}=0~\mbox{on}~ x=0.
\end{array}\end{equation}
Suppose that the intersection points of domain $\Omega$ with the $y$ positive semi-axis are $p$ and $q$, in turn. In addition, by the definition of exterior petal-like domain and Hopf lemma, we know that
$$\frac{\partial u}{\partial y}(p)>0~~\mbox{and}~\frac{\partial u}{\partial y}(q)<0.$$

By the continuity of function $\partial_{y} u$, then there is at least one point $\widetilde{p}$ between $p$ and $q$ on the $y$ positive semi-axis such that
$$\frac{\partial u}{\partial y}(\widetilde{p})=0.$$
Combining the above with (\ref{3.3}), we get that $u$ has at least one interior critical point on the $y$ positive semi-axis. Similarly, there is at least two critical points on the each other axis of symmetry. Thus we obtain that $u$ has at least $2k$ critical points along all symmetric axes of the exterior petal-like domain $\Omega$. \end{proof}

Now we will prove the exact number of critical points of the solution $u$ on each symmetric axis of domain $\Omega$, where $\Omega$ is an eccentric circle annular domain, a smooth exterior petal-like domain with even number axes of symmetry satisfied that the symmetric axes are perpendicular in pairs and the exterior boundary has nonzero curvature, and an annular domain where the interior and exterior boundaries are ellipses, respectively.

\begin{proof}[\bf Proof of Theorem B]  {\bf Case $B_1$:} If $\Omega$ is an eccentric circle annular domain. We set up the usual contradiction argument. Without loss of generality, we assume that $\Omega\subset \mathbb{R}^2$ is an eccentric circle annular domain, where the interior boundary $\gamma_I$ is a circle with its radius $r$ centered $(a,0)(a>0)$, and the exterior boundary $\gamma_E$ also is a circle with its radius $R$ centered (0,0) such that $a+r<R$, see Fig. 6. Let
$$v_1=\frac{\partial u}{\partial x}~\mbox{and}~ v_2=\frac{\partial u}{\partial y}.$$
So $v_1$ and $v_2$ satisfy the following equation respectively
$$\triangle v+f^\prime(u)v=0 ~\mbox{in}~\Omega.$$

By the symmetry of solution $u$, we know that $u(x,-y)=u(x,y)$, $\frac{\partial u}{\partial y}=0$ on $x$-axis, and the nodal set $v_1^{-1}(0)$ is symmetric with respect to $x$-axis. Similar to Theorem B, we know that the solution $u$ has at least one critical point on the coordinate intervals $(-R,a-r)$ and $(a+r,R)$ of $x$-axis, respectively. Next, we claim that $u$ has only one critical point $p_1$ and $p_2$ on the coordinate intervals $(-R,a-r)$ and $(a+r,R)$, respectively. In fact, if $u$ has more than one critical point on the coordinate intervals $(-R,a-r)$ and $(a+r,R)$, respectively. Without loss of generality, firstly we assume that $u$ has two critical points on the coordinate intervals $(a+r,R)$, that is $p_2, p^*_2$. The Case $A_1$ of Theorem A shows that the nodal set $v_1^{-1}(0)$ does not intersect with the line $T_a$ in $\Omega$. By Remark \ref{rem3.1} and Lemma \ref{le3.2}, then the possible distribution of the nodal set $v^{-1}_1(0)$ in the right of the line $T_a$ can see Fig. 5.

\begin{center}
  \includegraphics[width=5cm,height=5cm]{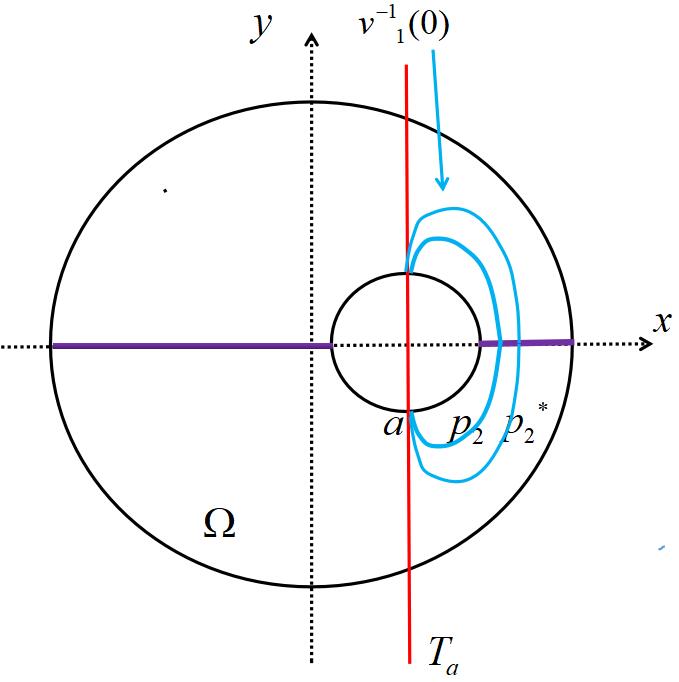}\\
  \scriptsize {\bf Fig. 5.}~~ The possible distribution of the nodal sets $v^{-1}_1(0)$ passing through the critical points $p_2, p^*_2$.
\end{center}
The nodal set $v^{-1}_1(0)$ known by Fig. 6 must form a connected domain in the right of the line $T_a$, by the analyticity of $f(u)$, we have $v_1\equiv 0$ in $\Omega$. However, the Case $A_1$ of Theorem A shows that the nodal set $v_1^{-1}(0)$ does not intersect with the line $T_a$ in $\Omega$, so this is a contradiction. Then $u$ has only one critical point on the coordinate interval $(a+r,R)$.

Next, by the section 3 in \cite{AlessandriniMagnanini1}, for the solution $u$ to (\ref{1.1}) in an eccentric circle annular domain $\Omega$, we know that
\begin{equation}\label{3.4}\begin{array}{l}
\sharp\{\mbox{maximum points of} ~u ~\mbox{in}~ \Omega\}=\sharp\{\mbox{saddle points of} ~u~ \mbox{in} ~\Omega\}.
\end{array}\end{equation}
We know that if $u$ has more than one critical point on the coordinate interval $(-R,a-r)$, then $u$ has at least three critical points on the coordinate interval $(-R,a-r)$. Similar to the proof on the coordinate interval $(a+r,R)$ of $x$-axis, we consider the possible distribution of the nodal set $v^{-1}_1(0)$ in the left of the line $T_a$, then the nodal set $v^{-1}_1(0)$ also must form a connected domain in the left of the line $T_a$. This is also a contradiction. Therefore $u$ also has only one critical point on the coordinate interval $(-R,a-r)$. So the possible distribution of  the nodal set $v^{-1}_1(0)$ in domain $\Omega$ can see Fig. 6.
\begin{center}
  \includegraphics[width=5cm,height=5cm]{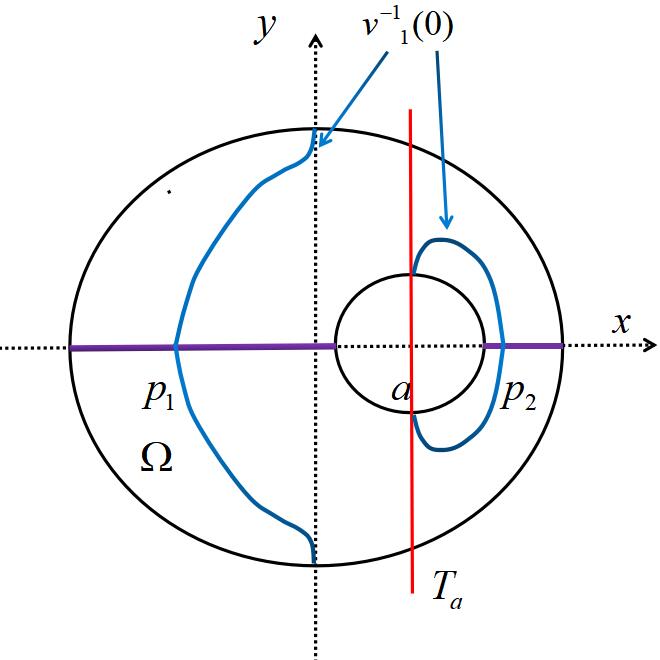}\\
  \scriptsize {\bf Fig. 6.}~~ The possible distribution of the nodal sets $v^{-1}_1(0)$ passing through the critical points $p_1,p_2$.
\end{center}
This completes the proof of Case $B_1$.

 {\bf Case $B_2$:} If $\Omega$ is a smooth exterior petal-like domain with even number axes of symmetry satisfied that the symmetric axes are perpendicular in pairs and the exterior boundary has nonzero curvature. Without loss of generality, we assume that $\Omega\subset \mathbb{R}^2$ is an exterior petal-like domain with the exterior boundary $\gamma_E$ is an ellipse. We set up the usual contradiction argument. According to the result of Theorem B, we know that $u$ has at least one critical point on each symmetric semi-axis. We set
$$v_1=\frac{\partial u}{\partial x}~\mbox{and}~ v_2=\frac{\partial u}{\partial y}.$$
Then $v_1$ and $v_2$ satisfy the following equations respectively
$$\triangle v_1+f^\prime(u)v_1=0 ~\mbox{in}~\Omega~\mbox{and}~ \triangle v_2+f^\prime(u)v_2=0 ~\mbox{in}~\Omega.$$

By the symmetry of domain $\Omega$, we know that $\frac{\partial u}{\partial x}=0$ on $x=0$, $\frac{\partial u}{\partial y}=0$ on $y=0$, and the nodal sets $v^{-1}_1(0)$ and $v^{-1}_2(0)$ are symmetric with respect to $x-$axis and $y-$axis respectively. Firstly, if $u$ has only one critical point on each symmetric semi-axis, that is $p_1,p_2,q_1,q_2$. By Remark \ref{rem3.1} and Lemma \ref{le3.2}, then the possible distribution of the nodal sets $v^{-1}_1(0)$ and $v^{-1}_2(0)$ can see Fig. 7.
\begin{center}
  \includegraphics[width=13cm,height=4.0cm]{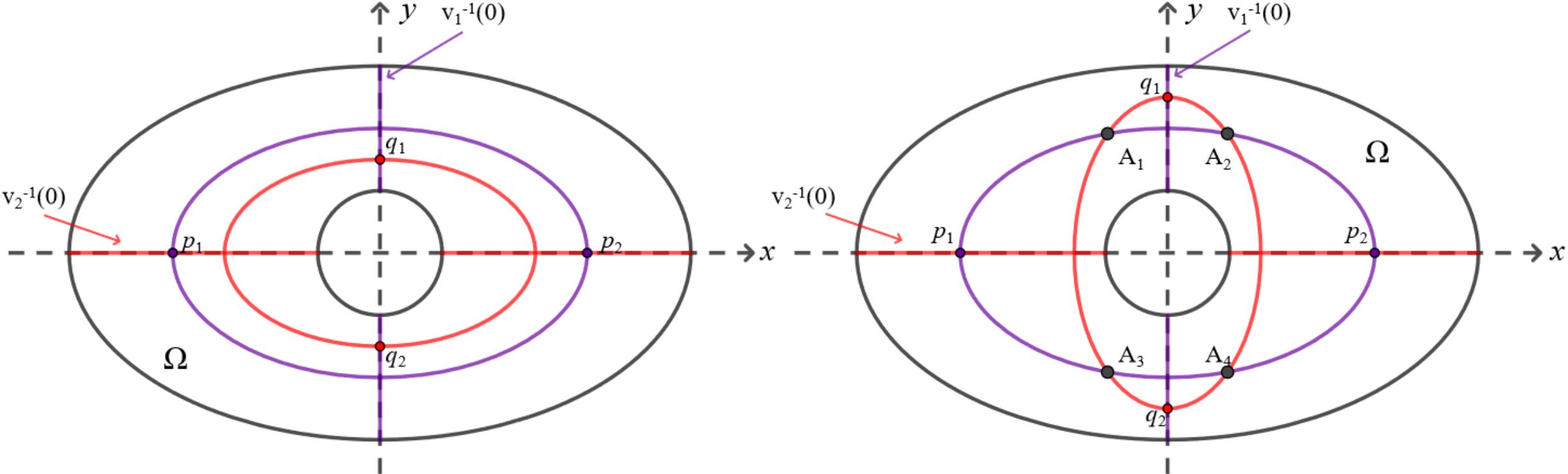}\\
  \scriptsize {\bf Fig. 7.}~~ The possible distribution of the nodal sets $v^{-1}_1(0)$ and $v^{-1}_2(0)$ passing through the critical points $p_1,p_2,q_1,q_2$.
\end{center}
Without loss of generality, we assume that there exist another critical points $p_4$ and $p_3$ on the $x$ positive and negative semi-axis respectively. Moreover $p_4$ and $p_3$ are symmetric with respect to $y$-axis. We know that $\mathcal{C}(u)=N_\alpha\cap N_\beta$ for any two noncollinear directions $\alpha,\beta\in\mathbb{S}^1.$ Now we consider the following possible distribution of the nodal set $v^{-1}_1(0)$ passing through the critical points $p_4$ and $p_3$, see Fig. 8.
\begin{center}
  \includegraphics[width=8.5cm,height=4.5cm]{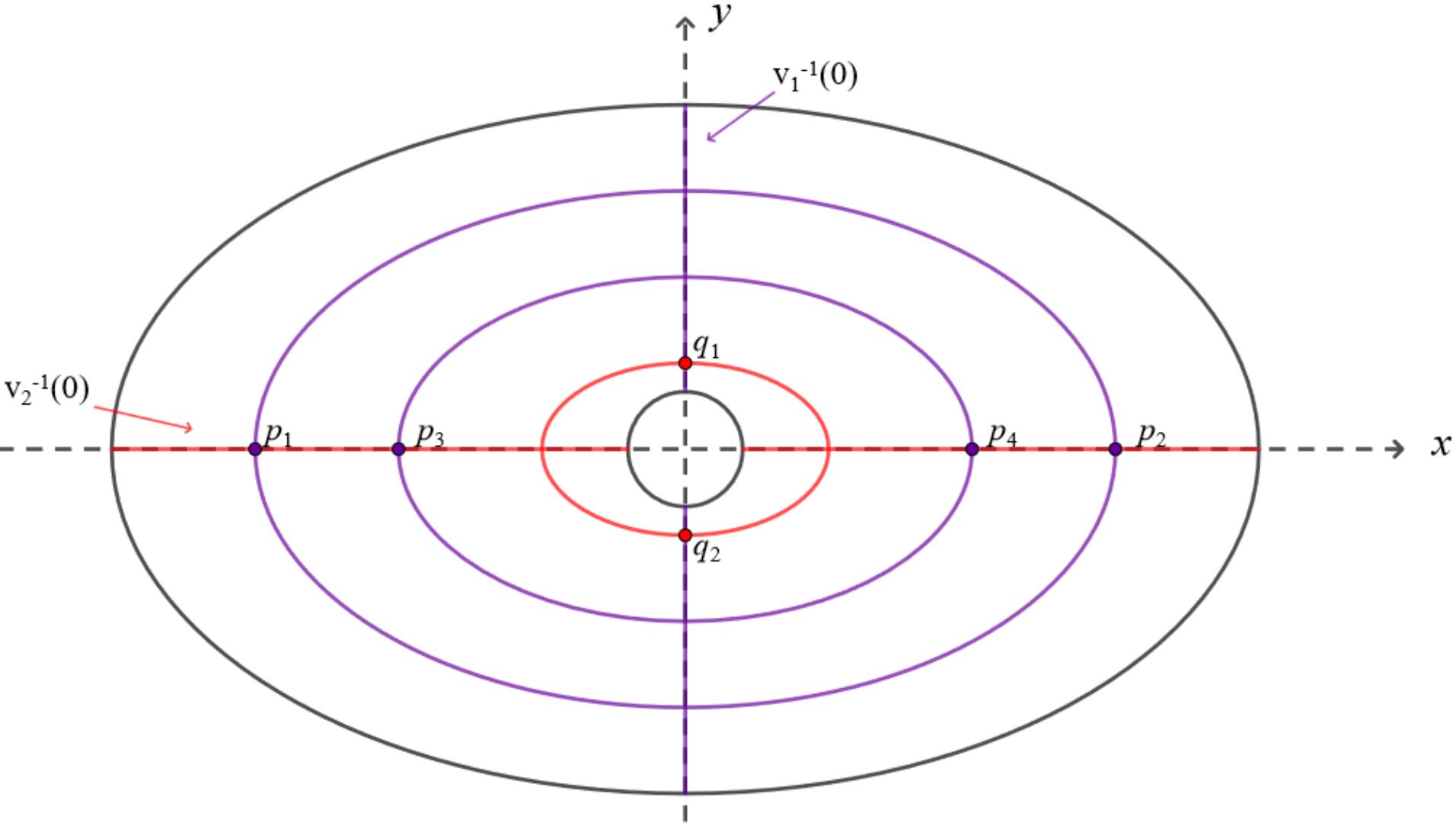}\\
  \scriptsize {\bf Fig. 8.}~~ The possible distribution of the nodal sets $v^{-1}_1(0)$ passing through the critical points $p_1,p_2,p_3,p_4$.
\end{center}
The nodal set  $v^{-1}_1(0)$ known by the Fig. 8 must form a connected domain. By the analyticity of function $f(u)$ , then $v_1=0$ in $\Omega.$ Similarly, by the result of $\sharp\{\mbox{maximum points of} ~u ~\mbox{in}~ \Omega\}=\sharp\{\mbox{saddle points of} ~u~ \mbox{in} ~\Omega\},$ we can obtain $v_2=0$ in $\Omega.$  So $u$ is a constant in $\Omega,$ this is a contradiction. Therefore the solution $u$ on the symmetric axes has exactly four critical points, i.e., one for each symmetric semi-axis. This completes the proof of Case $B_2$.

{\bf Case $B_3$:} If $\Omega$ is an annular domain where the interior and exterior boundaries are both ellipses. For this case, the idea of the proof is essentially the same as Case $B_2$. Here we omit the proof.  \end{proof}

By Theorem B, we have know that the solution $u$ on each symmetric axis has exactly two critical points in an eccentric circle annular domain or an exterior petal-like domain with the exterior boundary $\gamma_E$ is an ellipse. In order to prove Theorem C, we will further prove that solution $u$ has no extra critical points in any other sub-domain except on the symmetry axes, and give the specific distribution of the maximum points and saddle points.

\begin{proof}[\bf Proof of Theorem C] {\bf Case $C_1$:} If $\Omega$ is an eccentric circle annular domain. Firstly, we will  prove that solution $u$ has no extra critical points in any other sub-domain except on the symmetry axes. Without loss of generality, we assume that $\Omega\subset \mathbb{R}^2$ is an eccentric circle annular domain, where the interior boundary $\gamma_I$ is a circle with its radius $r$ centered $(a,0)(a>0)$, and the exterior boundary $\gamma_E$ also is a circle with its radius $R$ centered (0,0) such that $a+r<R$, see Fig. 9.
\begin{center}
  \includegraphics[width=5.5cm,height=5.5cm]{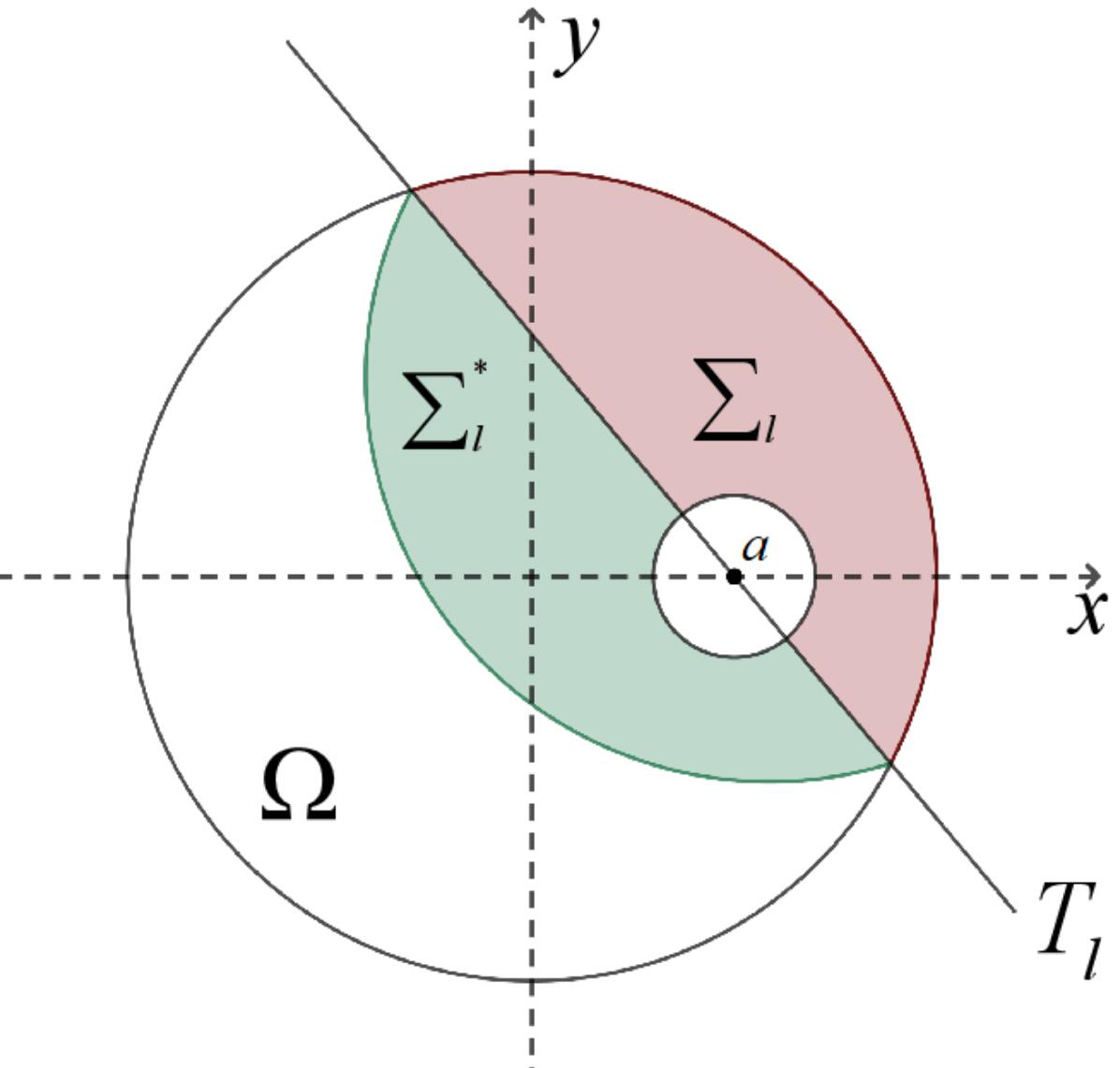}\\
  \scriptsize {\bf Fig. 9.}~~ The domains $\sum_l$ and $\sum^\ast_l$ .
\end{center}
Suppose that $l:y=k(x-a)~\mbox{for any}~k\neq 0,$ and whose right normal direction is $\nu_l^\perp$. (\ref{2.3}) means that solution $u$ has no critical points on $\{x=a\}\cap\Omega.$  Now we define
\begin{equation*}\label{}\begin{array}{l}
\left\{
\begin{array}{l}
\sum_l=\{p=(x,y)\in\Omega: p~\mbox{is in the~right~side~of~line~}l~\},\\
T_l=\{p=(x,y)\in\Omega: y=k(x-a)~\mbox{for any}~k\neq 0\},\\
\sum^\ast_l=\mbox{the reflection domain of}~\sum_l~\mbox{with respect to}~T_l,\\
p^\ast_l=(x^\ast,y^\ast)~\mbox{is the reflection point of }p\in\sum_l~\mbox{with respect to}~T_l.
\end{array}
\right.
\end{array}\end{equation*}

In $\sum_l$, we denote by
\begin{equation*}\label{}\begin{array}{l}
w_l(p)=u(p^\ast_l)-u(p)~\mbox{in}~\sum_l.
\end{array}\end{equation*}

By the differential mean value theorem, we have
\begin{equation*}\label{}\begin{array}{l}
\left\{
\begin{array}{l}
\triangle w_l+f^\prime(\kappa(p))w_l=0~\mbox{in}~\sum_l,\\
w_l>0~\mbox{on}~\partial\sum_l \cap \gamma_E,\\
w_l=0~\mbox{on}~\partial\sum_l\cap T_l,\\
w_l=0~\mbox{on}~\partial\sum_l\cap \gamma_I,
\end{array}
\right.
\end{array}\end{equation*}
where $\kappa(p)$ is a bounded function in $\sum_l$.

By the maximum principle and Hopf lemma, we have
\begin{equation}\label{3.5}\begin{array}{l}
\left\{
\begin{array}{l}
w_l>0,~\mbox{in}~\sum_l,\\
\frac{\partial w_l}{\partial\nu_l^\perp}|_{T_l}=-2\frac{\partial u}{\partial \nu_l^\perp}|_{T_l}<0.
\end{array}
\right.
\end{array}\end{equation}
This means that the solution $u$ can not have critical points on any $T_l,$ namely, the solution $u$ has no extra critical points in any sub-domains except on the symmetry axis. Furthermore, by Theorem B, we know that $u$ only on the symmetric axis has exactly two critical points, one of them is the maximum point and the other is the saddle point.

In fact, the result of $\sharp\{\mbox{maximum~points~of~}u\}-\sharp\{\mbox{saddle~points~of~}u\}=0$ in \cite{AlessandriniMagnanini1} implies that the number of maximum point and saddle point is equal. Next we will give the specific distribution of the maximum point and saddle point of solution $u$ in an eccentric circle annular domain. Use the previous notation, see Fig. 2, we define
\begin{equation*}\label{}\begin{array}{l}
\left\{
\begin{array}{l}
\Omega_a^+=\{p=(x,y)\in\Omega: x>a\},\\
\widetilde{\Omega}_a^+=\mbox{the reflection domain of}~\Omega_a^+~\mbox{with respect to the line}~T_a,\\
\widetilde{u}(x,y)=u(2a-x,y),~\mbox{where }(x,y)\in\Omega_a^+.
\end{array}
\right.
\end{array}\end{equation*}

Obviously, $\widetilde{u}$ satisfies the following boundary value problem in $\Omega_a^+$

\begin{equation}\label{3.6}\begin{array}{l}
\left\{
\begin{array}{l}
\Delta\widetilde{u}+f(\widetilde{u})=0,~\mbox{in}~\Omega_a^+,\\
\widetilde{u}|_{\partial\Omega_a^+\cap T_a}=u|_{\partial\Omega_a^+\cap T_a},\\
\widetilde{u}|_{\partial\Omega_a^+\cap \gamma_I}=0,\\
\widetilde{u}|_{\partial\Omega_a^+\cap \gamma_E}>0.
\end{array}
\right.
\end{array}\end{equation}

In addition, $u$ satisfies the following boundary value problem in $\Omega_a^+$
\begin{equation}\label{3.7}\begin{array}{l}
\left\{
\begin{array}{l}
\Delta u+f(u)=0,~\mbox{in}~\Omega_a^+,\\
u|_{\partial\Omega_a^+\cap T_a}=u|_{\partial\Omega_a^+\cap T_a},\\
u|_{\partial\Omega_a^+\cap \gamma_I}=0,\\
u|_{\partial\Omega_a^+\cap \gamma_E}=0.
\end{array}
\right.
\end{array}\end{equation}

For (\ref{3.6}) and (\ref{3.7}), using the differential mean value theorem and the maximum principle, we can get
$$\widetilde{u}-u>0~\mbox{in}~\Omega_a^+.$$

Therefore, we have
$$\sup_{\{x<a\}} u\geq\sup_{x\in\widetilde{\Omega}_a^+} u=\sup_{x\in\Omega_a^+} \widetilde{u}>\sup_{x\in\Omega_a^+} u.$$
This means that the maximum point and saddle point of the solution $u$ are distributed on the long symmetric semi-axis and short symmetric semi-axis,  respectively.

{\bf Case $C_2$:} If $\Omega$ is an exterior petal-like domain with the exterior boundary $\gamma_E$ is an ellipse. Similarly, we will prove that solution $u$ has no extra critical points in any other sub-domain except on the symmetry axes.  Without loss of generality, we consider the sub-domain in the first quadrant. Suppose that ray $l:y=kx, k=\tan\theta, \theta\in(0,\frac{\pi}{2}), x>0,$ and whose perpendicular through the origin is $T_l^\perp$, see Fig. 10.
\begin{center}
  \includegraphics[width=10cm,height=5.5cm]{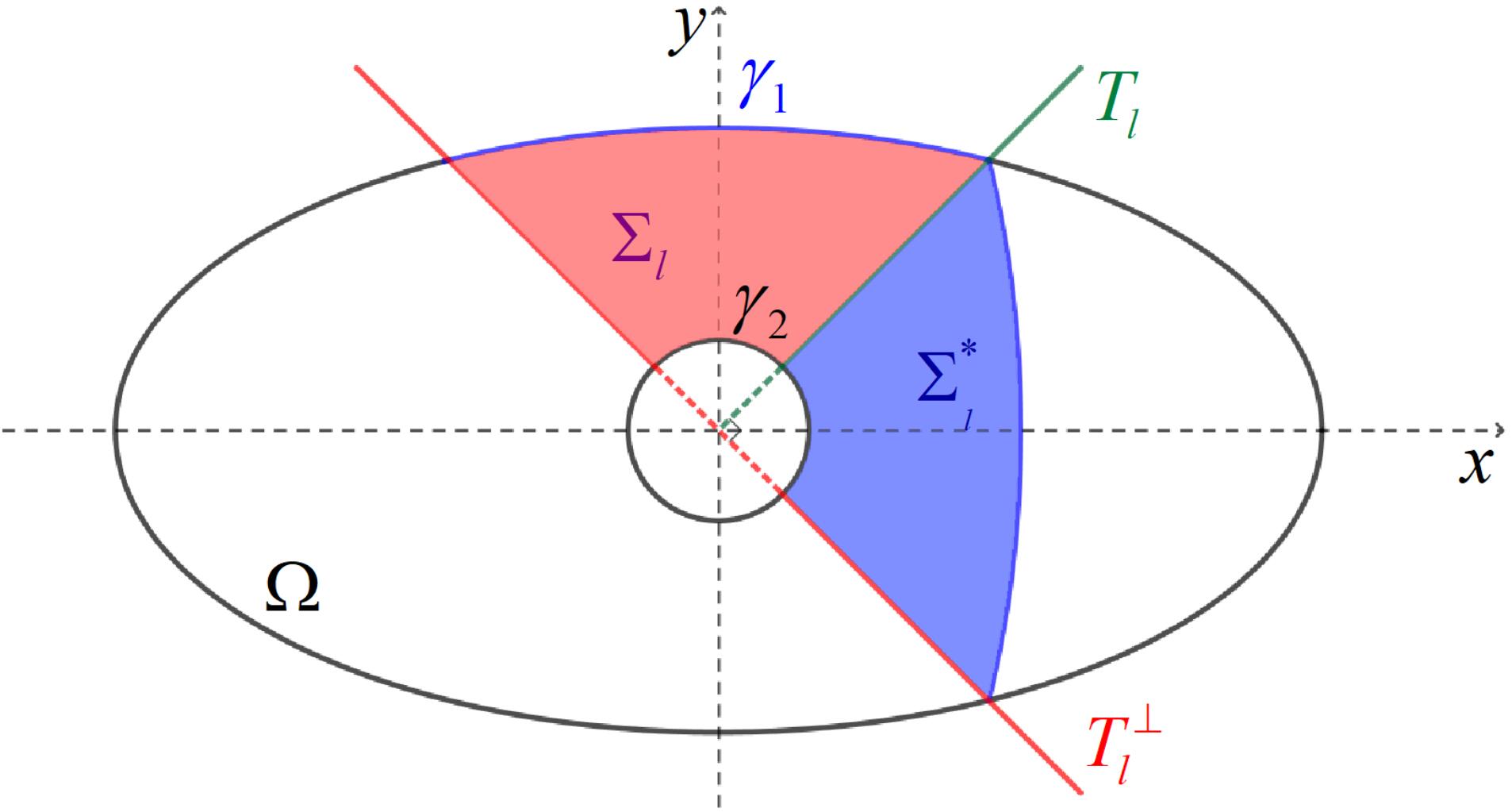}\\
  \scriptsize {\bf Fig. 10.}~~The domains $\sum_l$ and $\sum^\ast_l$.
\end{center}

Now we define
\begin{equation*}\label{}\begin{array}{l}
\left\{
\begin{array}{l}
\sum_l=\{p=(x,y)\in\Omega: y>kx~\mbox{and}~y>\frac{-1}{k}x\},\\
T_l=\{p=(x,y)\in\Omega: y=kx, k=\tan\theta, \theta\in(0,\frac{\pi}{2}), x>0\},\\
\sum^\ast_l=\mbox{the reflection domain of}~\sum_l~\mbox{with respect to}~T_l,\\
p^\ast_l=(x^\ast,y^\ast)~\mbox{is the reflection point of }p\in\sum_l~\mbox{with respect to}~T_l.
\end{array}
\right.
\end{array}\end{equation*}

In $\sum_l$, we denote by
\begin{equation*}\label{}\begin{array}{l}
v_l(p)=u(p^\ast_l)-u(p)~\mbox{in}~\sum_l.
\end{array}\end{equation*}

Then, using a proof similar to case $C_1$, we can reach that the solution $u$ can not have critical points on any $T_l,$ namely, the solution $u$ has no extra critical points in the first quadrant domain of $\Omega$. Similarly, $u$ also has no extra critical points in the other quadrant domain of $\Omega$. Furthermore, by Theorem B, we know that $u$ only on each symmetric axis has exactly two critical points, namely, $u$ has four critical points, two of them are the maximum points and the other two are the saddle points, and two maximum points and two saddle points of the solution $u$ are distributed on the long symmetric semi-axis and short symmetric semi-axis respectively. This completes the proof of Theorem C.\end{proof}

\begin{Remark}\label{rem3.4}
For some special exterior petal-like domain $\Omega$, namely, we make any ray $l$ through the origin of any asymmetric axis, and then make a vertical line $l^\perp$ of this ray  $l$ through the origin, if the small sub-domain surrounded by $l$,  $l^\perp$ and $\partial\Omega$ can be completely symmetric about ray $l$ into $\Omega$, the exact number of critical points of solution $u$ in such domains can be obtained by using the same method as above. We know that if such a domain $\Omega$ has $k$ ($k$ must be an even) symmetric axes, then $u$ has only $2k$ critical points in $\Omega$, and $k$ maximum points and $k$ saddle points of the solution $u$ are distributed on the long symmetric semi-axis and short symmetric semi-axis respectively. Theorem A and Theorem B show that the critical point set of the solution $u$ is unstable when the exterior boundary $\gamma_E$ of planar concentric circle annular domain occurs small deformation or minor movement. The critical point set $\mathcal{C}(u)$ changes from a Jordan curve to finitely many isolated points.
\end{Remark}

\begin{Remark}\label{rem3.5}
If the condition of $f(u)\geq 0$ is replaced by the condition of $f(u)\leq 0$ in all theorems, the above all conclusions still hold. In addition, by the strong maximum principle, we know that the solutions to (\ref{1.1}) in Theorem A, Theorem B and Theorem C do not have minimum points. The result of $\sharp\{\mbox{maximum~points~of~u}\}-\sharp\{\mbox{saddle~points~of~u}\}=0$ in \cite{AlessandriniMagnanini1} implies that the number of maximum points and saddle points is equal in Theorem A, Theorem B and Theorem C .
\end{Remark}

\section{Geometric location of critical points}

~~~~In this section, we will give the geometric location of the critical point sets, that is, Theorem D. Firstly, we will present the proof of the geometric location of critical points of a solution $u$ to (\ref{1.1}) in a particular exterior petal-like domain $\Omega\subset \mathbb{R}^2$ by moving plane method and moving sphere method, where the radius of the interior boundary $\gamma_I$ is $a$, and the length of each semi-axis of the exterior boundary $\gamma_E$ is $b_i$ ($i=1,2$) respectively.

\begin{proof}[\bf Proof of (1) of Theorem D] We need to give the geometric location of critical points using the information of the interior and exterior boundaries respectively. We divide the proof into two steps.

{\bf Step 1.} Firstly, we use the exterior boundary information to give the location distribution of critical points by moving plane method, that is
 \begin{equation}\label{4.1}\begin{array}{l}
\mathcal{C}(u)\subset\big\{(-\frac{a+b_1}{2},\frac{a+b_1}{2})\times(-\frac{a+b_2}{2},\frac{a+b_2}{2})\big\}\cap\Omega.
 \end{array}\end{equation}

Without loss of generality, we only prove the case on the positive $x$ semi-axis. Write $p=(x,y)\in \Omega$ for any $y\in \mathbb{R}.$ For $\lambda\in [\frac{a+b_1}{2},b_1)$, we define
\begin{equation*}\label{}\begin{array}{l}
\left\{
\begin{array}{l}
\sum_\lambda=\{p=(x,y)\in\Omega: x>\lambda\},\\
T_\lambda=\{p=(x,y)\in\Omega: x=\lambda\},\\
\sum^\prime_\lambda=\mbox{the~reflection~domain~of} ~\sum_\lambda \mbox{with~ respect~ to}~T_\lambda ,\\
p_\lambda=(2\lambda-x,y).
\end{array}
\right.
\end{array}\end{equation*}

In domain $\sum_\lambda$, we set
 \begin{equation*}\label{}\begin{array}{l}
 w_\lambda(p)=u(p)-u(p_\lambda)~\mbox{for~any}~p\in\sum_\lambda.
 \end{array}\end{equation*}

By the differential mean value theorem and  (\ref{1.1}), we obtain that
\begin{equation*}\label{}\begin{array}{l}
\left\{
\begin{array}{l}
\triangle w_\lambda+c(x,\lambda)w_\lambda=0~\mbox{in}~\sum_\lambda,\\
w_\lambda<0~\mbox{on}~\partial\sum_\lambda \setminus T_\lambda,\\
w_\lambda\equiv 0~\mbox{on}~\partial\sum_\lambda\cap T_\lambda,
\end{array}
\right.
\end{array}\end{equation*}
where $c(x,\lambda)$ is a bounded function in $\sum_\lambda$.

Next we claim that
 \begin{equation}\label{4.2}\begin{array}{l}
 w_\lambda<0~\mbox{in}~\sum_\lambda~\mbox{for~any}~ \lambda\in[\frac{a+b_1}{2},b_1).
 \end{array}\end{equation}
This implies in particular that $w_\lambda$ obtains along $\partial\sum_\lambda\cap T_\lambda$ its maximum in $\overline{\sum_\lambda}$.

For any $\lambda$ close to $b_1$, by Proposition 2.13 (the maximum principle for a narrow domain) in \cite{Han}, we have $w_\lambda<0.$ Suppose that $[\lambda_0,b_1)$ is the largest interval of values of $\lambda$ such that $w_\lambda<0$ in $\sum_\lambda$. Next we want to prove that $\lambda_0=\frac{a+b_1}{2}.$ If $\lambda_0>\frac{a+b_1}{2}$, by continuity, we have $w_{\lambda_0}\leq 0$ in $\sum_{\lambda_0}$ and $w_{\lambda_0}\not\equiv 0$ on $\partial\sum_{\lambda_0}$. Serrin's comparison principle \cite{Serrin} implies that $w_{\lambda_0}< 0$ in $\sum_{\lambda_0}$. We will prove that
\begin{equation*}\label{}\begin{array}{l}
w_{\lambda_0-\varepsilon}< 0 ~\mbox{in}~\sum_{\lambda_0-\varepsilon},
 \end{array}\end{equation*}
for any small $\varepsilon>0$.

Fix $\sigma>0$ (the following will be determined). Let $\mathbb{W}$ be a closed subset in $\sum_{\lambda_0}$ such that $|\sum_{\lambda_0}\backslash\mathbb{W}|<\frac{\sigma}{2}$. Then $w_{\lambda_0}< 0$ in $\sum_{\lambda_0}$ implies that

\begin{equation*}\label{}\begin{array}{l}
w_{\lambda_0}(p)\leq -\alpha<0~\mbox{for~any}~p\in\mathbb{W}~\mbox{and~some}~\alpha>0.
 \end{array}\end{equation*}
According to the continuity, we get
\begin{equation}\label{4.3}\begin{array}{l}
w_{\lambda_0-\varepsilon}< 0 ~\mbox{in}~\mathbb{W}.
 \end{array}\end{equation}
We choose $\varepsilon>0$ to be small enough that$|\sum_{\lambda_0-\varepsilon}\backslash\mathbb{W}|<\sigma$ holds. We choose $\sigma$ in such a way that we can apply  Theorem 2.32 (the maximum principle for a domain with small volume) in \cite{Han} to $w_{\lambda_0-\varepsilon}$ in $\sum_{\lambda_0-\varepsilon}\backslash\mathbb{W}$. Then we have that
 \begin{equation*}\label{}\begin{array}{l}
 w_{\lambda_0-\varepsilon}\leq 0~\mbox{in}~ \sum_{\lambda_0-\varepsilon}\backslash\mathbb{W}.
 \end{array}\end{equation*}
 In addition, by Serrin's comparison principle \cite{Serrin}, we achieve that
 \begin{equation}\label{4.4}\begin{array}{l}
 w_{\lambda_0-\varepsilon}< 0~\mbox{in}~ \sum_{\lambda_0-\varepsilon}\backslash\mathbb{W}.
 \end{array}\end{equation}
By (\ref{4.3}) and (\ref{4.4}), this contradicts the choice of $\lambda_0$, therefore
\begin{equation*}\label{}\begin{array}{l}
 w_\lambda<0~\mbox{in}~\sum_\lambda~\mbox{for~any}~ \lambda\in[\frac{a+b_1}{2},b_1).
 \end{array}\end{equation*}

By $w_\lambda\equiv 0$ on $\partial\sum_\lambda\cap T_\lambda$ and Hopf Lemma, we obtain that
\begin{equation*}\label{}\begin{array}{l}
D_{x}w_\lambda|_{x=\lambda}=2D_{x}u|_{x=\lambda}<0 ~\mbox{for~any}~\lambda\in[\frac{a+b_1}{2},b_1).
\end{array}\end{equation*}
This means that the solution $u$ does not exist any critical points in domain  $\{[\frac{a+b_1}{2},b_1)\times(-b_2,b_2)\}\cap\Omega$. Similarly, we can prove the case on the negative $x$ semi-axis, positive and negative $y$ semi-axis, respectively. In this step, we show that all the critical points of solution $u$ are distributed within the following domain $$\big\{(-\frac{a+b_1}{2},\frac{a+b_1}{2})\times(-\frac{a+b_2}{2},\frac{a+b_2}{2})\big\}\cap\Omega.$$

{\bf Step 2.} Next, we present the geometric location of critical points using the interior boundary information by the moving sphere method, a variant of the moving plane method, namely
\begin{equation}\label{4.5}\begin{array}{l}
\mathcal{C}(u)\subset\Big\{\Omega\setminus\big\{(x,y):a<\sqrt{x^2+y^2}\leq \sqrt{ac}\big\}\Big\},
\end{array}\end{equation}
where $c$ is the radius of inner cutting ball (circle) of $\Omega_{\gamma_E}$, and $\Omega_{\gamma_E}$ is the domain bounded by the exterior boundary $\gamma_E$.

Write $p=(x,y)\in \Omega\subset\mathbb{R}^2$. For $\lambda\in (a,\sqrt{ac}]$, we define
\begin{equation*}\label{}\begin{array}{l}
\left\{
\begin{array}{l}
\sum^\lambda=\{p=(x,y)\in\Omega: a<|p|<\lambda\},\\
T^\lambda=\{p=(x,y)\in\Omega: |p|=\lambda\},\\
p^\lambda=\frac{\lambda^2p}{|p|^2},\\
{\sum^{\lambda}}^\prime=\{p^\lambda:\mbox{the~Kelvin~transform~domain~of} ~\sum^\lambda \mbox{with~respect~to}~T^\lambda,~\mbox{where}~ p\in\sum^\lambda\}.
\end{array}
\right.
\end{array}\end{equation*}

In domain $\sum^\lambda$, after some calculations \cite{Li2004}, we have
$$\triangle_p u(p^\lambda)+\frac{\lambda^4}{|x|^4}f(u(p^\lambda))=0.$$

According to (\ref{1.1}), we get
\begin{equation*}\label{}\begin{array}{l}
\begin{array}{l}
\triangle_p(u(p)-u(p^\lambda))+f(u(p))-\frac{\lambda^4}{|x|^4}f(u(p^\lambda))=0~\mbox{in}~\sum^\lambda,
\end{array}
\end{array}\end{equation*}
it means that
\begin{equation*}\label{}\begin{array}{l}
\begin{array}{l}
\triangle_p(u(p)-u(p^\lambda))+f(u(p))-f(u(p^\lambda))+(1-\frac{\lambda^4}{|x|^4})f(u(p^\lambda))=0~\mbox{in}~\sum^\lambda.
\end{array}
\end{array}\end{equation*}

By the definition of $\sum^\lambda$ and the condition of $f(u)\geq 0$, we obtain that
\begin{equation}\label{4.6}\begin{array}{l}
\begin{array}{l}
\triangle_p(u(p)-u(p^\lambda))+f(u(p))-f(u(p^\lambda))\geq 0~\mbox{in}~\sum^\lambda.
\end{array}
\end{array}\end{equation}

Now we define
 \begin{equation*}\label{}\begin{array}{l}
\psi_\lambda(p)=u(p)-u(p^\lambda)~\mbox{for~any}~p\in\sum^\lambda.
 \end{array}\end{equation*}

By the differential mean value theorem and  (\ref{1.1}), we have that
\begin{equation*}\label{}\begin{array}{l}
\left\{
\begin{array}{l}
\triangle_p \psi_\lambda+d(x,\lambda)\psi_\lambda\geq 0~\mbox{in}~\sum^\lambda,\\
\psi_\lambda<0~\mbox{on}~\partial\sum^\lambda \setminus T^\lambda,\\
\psi_\lambda\equiv 0~\mbox{on}~\partial\sum^\lambda\cap T^\lambda,
\end{array}
\right.
\end{array}\end{equation*}
where $d(x,\lambda)$ is a bounded function in $\sum^\lambda$.

Next we claim that
 \begin{equation}\label{4.7}\begin{array}{l}
 \psi_\lambda<0~\mbox{in}~\sum^\lambda~\mbox{for~any}~ \lambda\in(a,\sqrt{ac}].
 \end{array}\end{equation}
This implies in particular that $\psi_\lambda$ obtains along $\partial\sum^\lambda\cap T^\lambda$ its maximum in $\overline{\sum^\lambda}$.

For $\lambda$ close to $a$, by Theorem 2.32 (the maximum principle for a domain with small volume) in \cite{Han}, we have $\psi_\lambda<0.$ Suppose that $(a,\lambda_0]$ is the largest interval of values of $\lambda$ such that $\psi_\lambda<0$ in $\sum^\lambda$. Next we need prove that $\lambda_0=\sqrt{ac}.$ If $\lambda_0<\sqrt{ac}$, by continuity, we have that $\psi_{\lambda_0}\leq 0$ in $\sum^{\lambda_0}$ and $\psi_{\lambda_0}\not\equiv 0$ on $\partial\sum^{\lambda_0}$. By Serrin's comparison principle, we have $\psi_{\lambda_0}< 0$ in $\sum^{\lambda_0}$.
Similar to the proof of (\ref{4.2}), by the maximum principle for a domain with small volume \cite{Han} and Serrin's comparison principle \cite{Serrin}, we can easily obtain
\begin{equation}\label{4.8}\begin{array}{l}
\psi_{\lambda_0+\varepsilon}< 0 ~\mbox{in}~\sum^{\lambda_0+\varepsilon},
 \end{array}\end{equation}
for any small $\varepsilon>0$. This contradicts the choice of $\lambda_0$, then we have
\begin{equation*}\label{}\begin{array}{l}
\psi_\lambda<0~\mbox{in}~\sum^\lambda~\mbox{for~any}~ \lambda\in(a,\sqrt{ac}].
 \end{array}\end{equation*}

By $\psi_\lambda\equiv 0$ on $\partial\sum^\lambda\cap T^\lambda$ and Hopf Lemma, we have
$$\frac{\partial \psi_\lambda(p)}{\partial\nu}|_{|p|=\lambda}>0,$$
and noting that
\begin{equation*}\label{}\begin{array}{l}
\frac{\partial u(p^\lambda)}{\partial\nu}|_{|p^\lambda|=\lambda}=-\frac{\partial u(p)}{\partial\nu}|_{|p|=\lambda},~\frac{\partial u(p)}{\partial\nu}|_{|p|=\lambda}=\frac{1}{2}\frac{\partial \psi_\lambda(p)}{\partial\nu}|_{|p|=\lambda}>0,
\end{array}\end{equation*}
for any $\lambda\in(a,\sqrt{ac}],$ where $\nu$ is the outward normal vector.

 Hence the solution $u$ does not exist any critical points in domain $\big\{(x,y):a<\sqrt{x^2+y^2}\leq \sqrt{ac}\big\}$. In this step, we show that all the critical points of solution $u$ are distributed within the following domain
 $$\Big\{\Omega\setminus\big\{(x,y):a<\sqrt{x^2+y^2}\leq \sqrt{ac}\big\}\Big\}.$$

 Combining (\ref{4.1}) and (\ref{4.5}) , we show that the critical point set
  \begin{eqnarray*}\label{}
\mathcal{C}(u)\subset\big\{(-\frac{a+b_1}{2},\frac{a+b_1}{2})\times(-\frac{a+b_2}{2},\frac{a+b_2}{2})\big\} \\
\cap\Big\{\Omega\setminus\big\{(x,y):a<\sqrt{x^2+y^2}\leq \sqrt{ac}\big\}\Big\}.
\end{eqnarray*}\end{proof}

\noindent {\bf Example 1.} {\it The geometric distribution of level sets of the solution to the following boundary problem, see Fig. 11.}
\begin{equation*}\label{}\begin{array}{l}
\left\{
\begin{array}{l}
-\triangle u=1~\mbox{in}~\Omega=\big\{(x,y):\{\frac{x^2}{36}+\frac{y^2}{16}<1\}\setminus \{x^2+y^2\leq 1\} \big\},\\
u=0~\mbox{on}~\{(x,y):\frac{x^2}{36}+\frac{y^2}{16}=1\},\\
u=0~\mbox{on}~\{(x,y):x^2+y^2=1\}.
\end{array}
\right.
\end{array}\end{equation*}
\begin{center}
  \includegraphics[width=8cm,height=5.5cm]{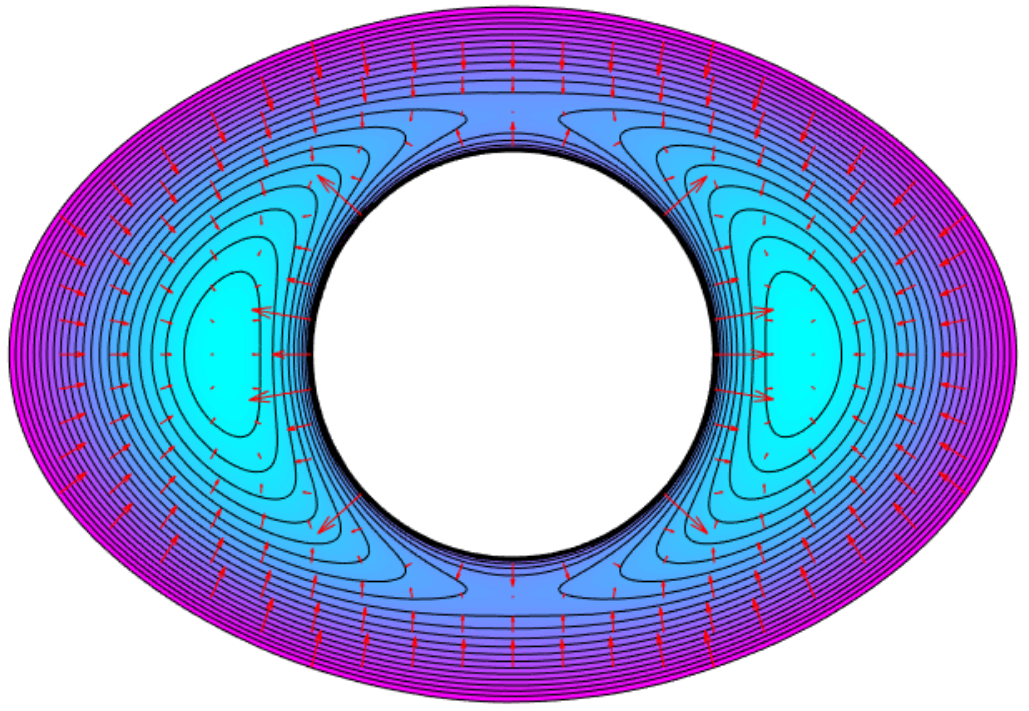}\\
  \scriptsize {\bf Fig. 11.}~~Top view of the geometric distribution of level sets of solution $u$ in an exterior petal-like domain $\Omega$ .
\end{center}

In the rest of this section, we will present the proof of the geometric location of critical points of a solution $u$ in an eccentric circle annular domain $\Omega\subset \mathbb{R}^2$, where the interior boundary $\gamma_I$ is a circle with its radius $r$ centered at $(a,0)(a>0)$, and the exterior boundary $\gamma_E$ also is a circle with its radius $R$ centered at (0,0) such that $a+r<R$.

\begin{proof}[\bf Proof of (2) of Theorem D] We will give the geometric location of critical points using the information of the interior and exterior boundaries, respectively. We divide the proof into two steps.

{\bf Step 1.} Firstly, similar to the step 1 of the proof of (1) of Theorem D, we use the exterior boundary information to give the location distribution of critical points by moving plane method, namely
 \begin{equation}\label{4.9}\begin{array}{l}
\mathcal{C}(u)\subset\big\{(-\frac{-R+a-r}{2},\frac{a+r+R}{2})\times(-R,R)\big\}\cap\Omega.
 \end{array}\end{equation}

{\bf Step 2.} Similar to the step 2 of the proof of (1) of Theorem D, we present the location distribution of critical points using the interior boundary information by the moving sphere method, a variant of the moving plane method, that is
\begin{equation}\label{4.10}\begin{array}{l}
\mathcal{C}(u)\subset\Big\{\Omega\setminus\big\{(x,y):r<\sqrt{(x-a)^2+y^2}\leq \sqrt{r(R-a)}\big\}\Big\}.
\end{array}\end{equation}

 Combining (\ref{4.9}) and (\ref{4.10}), we show that the critical point set
\begin{equation*}\label{}\begin{array}{l} \mathcal{C}(u)\subset\Big\{(\frac{-R+a-r}{2},\frac{a+r+R}{2})\times(-R,R)\Big\}\cap\\ ~~~~~~~~~~\Big\{\Omega\setminus\big\{(x,y):r<\sqrt{(x-a)^2+y^2}\leq \sqrt{r(R-a)}\big\}\Big\}.
\end{array}\end{equation*}
\end{proof}

\noindent {\bf Example 2.} {\it We consider the geometric distribution of level sets of the solution to the following boundary problem, see Fig. 12.}
\begin{equation*}\label{}\begin{array}{l}
\left\{
\begin{array}{l}
-\triangle u=1~\mbox{in}~\Omega=\big\{(x,y):\{x^2+y^2\leq (0.8)^2\}\setminus \{(x-0.3)^2+y^2\leq (0.2)^2\} \big\},\\
u=0~\mbox{on}~\{(x,y):x^2+y^2=(0.8)^2\},\\
u=0~\mbox{on}~\{(x,y):(x-0.3)^2+y^2=(0.2)^2\}.
\end{array}
\right.
\end{array}\end{equation*}
\begin{center}
  \includegraphics[width=5cm,height=5cm]{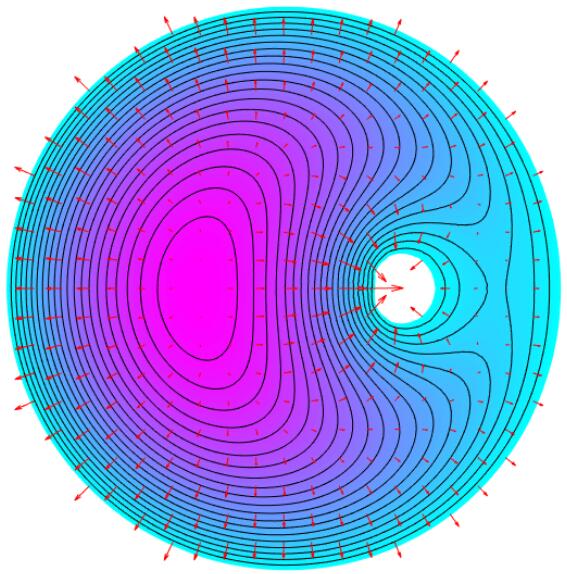}\\
  \scriptsize {\bf Fig. 12.}~~Top view of the geometric distribution of level sets of solution $u$ in an eccentric circle annular domain $\Omega$.
\end{center}

\begin{Remark}\label{rem4.1}
For general exterior petal-like domains $\Omega\subset\mathbb{R}^n(n\geq 2)$ and eccentric ball annular domains $\Omega\subset\mathbb{R}^n(n\geq 3)$, we can also use the interior and exterior boundary information to obtain the geometric location of the critical point set by moving plane method and moving sphere method, respectively.
\end{Remark}

\noindent  \textbf{Acknowledgement.} The first author is grateful to Jun Zou for his helpful conversations about Theorem C.

\end{document}